\documentclass[11pt]{article}
\usepackage{amsmath,amssymb,amsthm,bm,graphicx,color,epsfig}

\newtheorem{theorem}{Theorem}[section]
\newtheorem{lemma}[theorem]{Lemma}

\newcommand{\R}{\mathbb{R}}
\renewcommand{\i}{\imath}

\newcommand{\eps}{\varepsilon}
\newcommand{\p}{\partial}

\numberwithin{equation}{section}

\begin{document}

\title{A Fast Butterfly Algorithm for the Computation of Fourier
  Integral Operators}

\author{Emmanuel J. Cand\`es$^{\dagger}$
\and Laurent Demanet$^{\sharp}$
\and Lexing Ying$^{\S}$\\
  \vspace{-.1cm}\\
  $\dagger$ Applied and Computational Mathematics,
  Caltech, Pasadena, CA 91125\\
  \vspace{-.3cm}\\
  $\sharp$  Department of Mathematics, Stanford University,
  Stanford, CA 94305\\
  \vspace{-.3cm}\\
  $\S$ Department of Mathematics, University of Texas,
  Austin, TX 78712}

\date{August 2008}
\maketitle

\begin{abstract}
  This paper is concerned with the fast computation of Fourier
  integral operators of the general form $\int_{\R^d} e^{2\pi\i
    \Phi(x,k)} f(k) d k$, where $k$ is a frequency variable,
  $\Phi(x,k)$ is a phase function obeying a standard homogeneity
  condition, and $f$ is a given input. This is of interest for such
  fundamental computations are connected with the problem of finding
  numerical solutions to wave equations, and also frequently arise in
  many applications including reflection seismology, curvilinear
  tomography and others.  In two dimensions, when the input and output
  are sampled on $N \times N$ Cartesian grids, a direct evaluation
  requires $O(N^4)$ operations, which is often times prohibitively expensive. 

  This paper introduces a novel algorithm running in $O(N^2 \log N)$
  time, i.~e.~with near-optimal computational complexity, and whose
  overall structure follows that of the butterfly algorithm
  \cite{michielssen-1996-mmda}.  Underlying this algorithm is a
  mathematical insight concerning the restriction of the kernel
  $e^{2\pi\i \Phi(x,k)}$ to subsets of the time and frequency
  domains. Whenever these subsets obey a simple geometric condition,
  the restricted kernel has approximately low-rank; we propose
  constructing such low-rank approximations using a special
  interpolation scheme, which prefactors the oscillatory component,
  interpolates the remaining nonoscillatory part and, lastly,
  remodulates the outcome. A byproduct of this scheme is that the
  whole algorithm is highly efficient in terms of memory
  requirement. Numerical results demonstrate the performance and
  illustrate the empirical properties of this algorithm.
\end{abstract}



%----------------------------------------------------------
\section{Introduction}
\label{sec:intro}

%------------

This paper introduces an efficient algorithm for evaluating discrete
Fourier integral operators. Let $N$ be a positive integer, which is
assumed to be an integer power of 2 with no loss of generality, and
define the Cartesian grids $X = \{(i_1/N, i_2/N), 0\le i_1, i_2 <
N\}$ and $\Omega = \{(k_1, k_2), -N/2 \le k_1 , k_2 <N/2\}$. A
discrete Fourier integral operator (FIO) with constant amplitude is
defined by
\begin{equation}
  u(x) = \sum_{k\in \Omega} e^{2\pi\i \Phi(x,k)} f(k), \quad  x\in X,
  \label{eq:dfio}
\end{equation}
where $\{f(k), k\in \Omega\}$ is a given input, $\{u(x), x\in X\}$ is
the output and as usual, $\i = \sqrt{-1}$. By an obvious analogy with
problems in electrostatics, it will be convenient throughout the paper
to refer to $\{f(k), k\in \Omega\}$ as {\em sources} and $\{u(x), x\in
X\}$ as {\em potentials}. Here, the {\em phase function} $\Phi(x,k)$
is assumed to be smooth in $(x,k)$ for $k\not=0$ and obeys an
homogeneity condition of degree 1 in $k$, namely, $\Phi(x,\lambda k) =
\lambda \Phi(x,k)$ for each $\lambda > 0$.

A direct numerical evaluation of \eqref{eq:dfio} at all the points in
$X$ takes $O(N^4)$ flops, which can be very expensive for large values
of $N$. Surveying the literature, the main obstacle to constructing
fast algorithms for \eqref{eq:dfio} is the oscillatory behavior of the
kernel $e^{2\pi\i \Phi(x,k)}$ when $N$ is large, which prevents the
use of the standard multiscale techniques developed in
\cite{beylkin-1991-fwtna,
  boerm-2003-hm,greengard-1987-afaps,hackbusch-1989-fmubempc}.
Against this background, the contribution of this paper is to
introduce a novel algorithm running in $O(N^2\log N)$ operations,
where the constant is polylogarithmic in the prescribed accuracy
$\eps$.

\subsection{General strategy}

Because the phase function $\Phi(x,k)$ is singular at $k=0$, the first
step consists in representing the frequency variable $k$ in polar
coordinates via the transformation
\begin{equation}
  k = (k_1,k_2) = \frac{\sqrt{2}}{2} N  p_1 e^{2\pi\i p_2}, \quad 
e^{2\pi\i p_2}  = (\cos 2\pi p_2, \sin 2\pi p_2).
  \label{eq:polar}
\end{equation}
Here and below, the set of all possible points $p$ generated from
$\Omega$ is denoted by $P$, see Figure \ref{fig:ptree}(b). Note that
this transformation guarantees that each point $p=(p_1,p_2)$ belongs
to the unit square $[0,1]^2$ since $-N/2 \le k_1, k_2 < N/2$. Because
of the homogeneity of $\Phi$, the phase function $\Phi$ may be
expressed in polar coordinates as
\[
\Phi(x,k) = N \, \frac{\sqrt{2}}{2} \Phi\left(x, e^{2\pi\i p_2}
\right) p_1 := N \, \Psi(x,p). 
\]
Since $\Phi(x,k)$ is smooth in $(x,k)$ for $k\not=0$, $\Psi(x,p)$ is a
smooth function of $(x,p)$ with $x$ and $p$ in $[0,1]^2$.

\begin{figure}[h]
  \begin{center}
    \begin{tabular}{ccc}
      \includegraphics[height=1.8in]{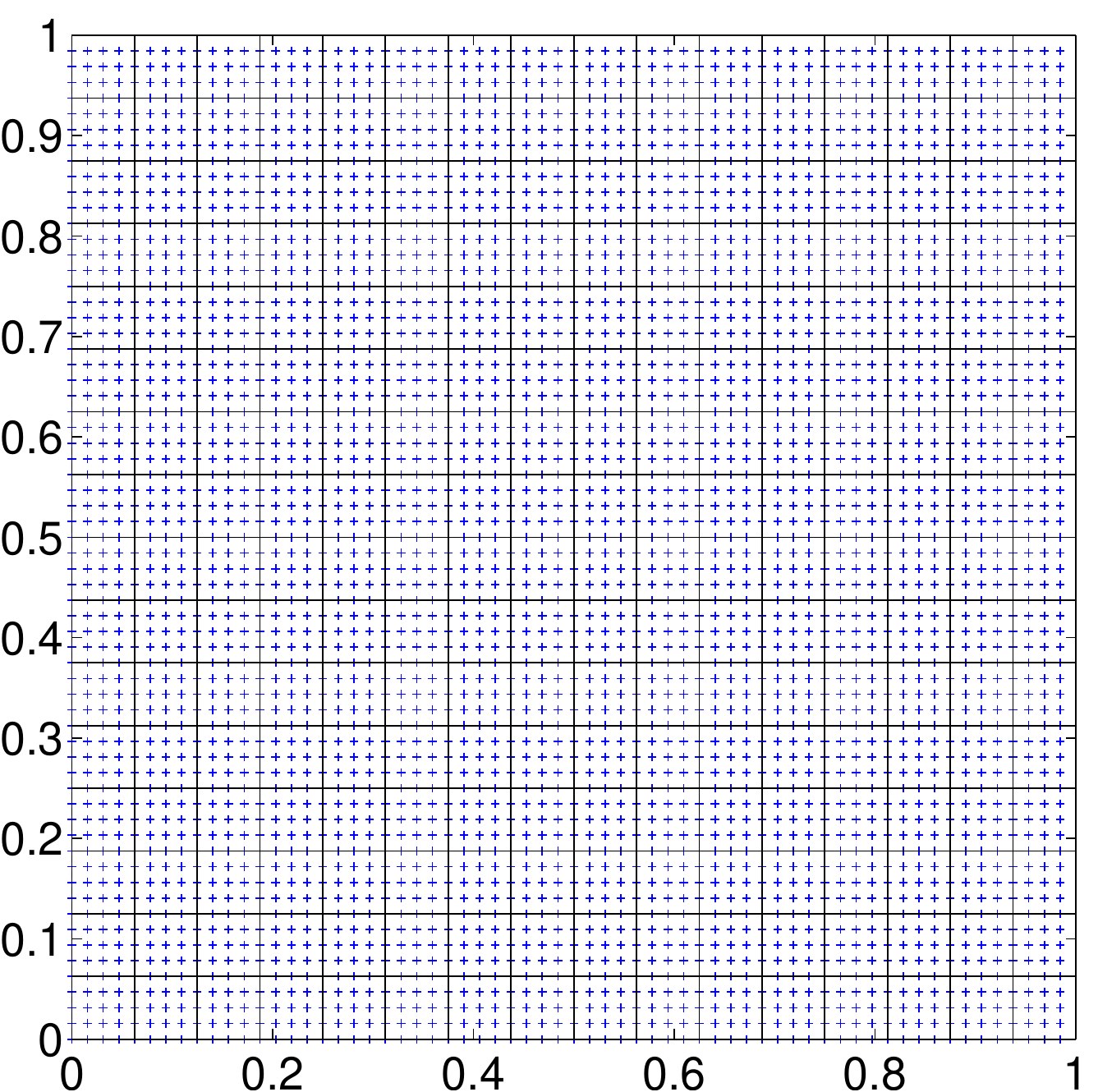} &
      \includegraphics[height=1.8in]{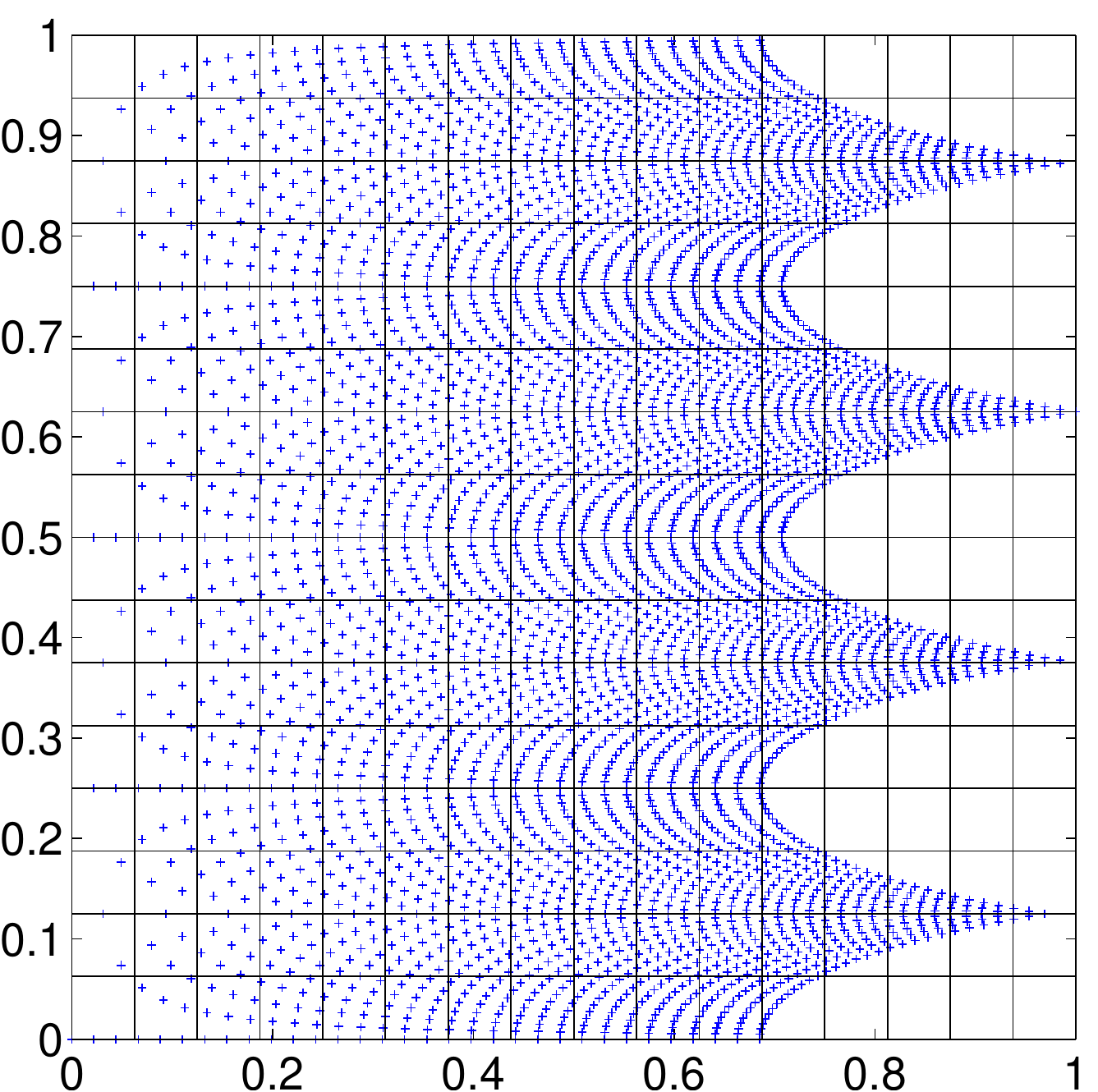} &
      \includegraphics[height=1.8in]{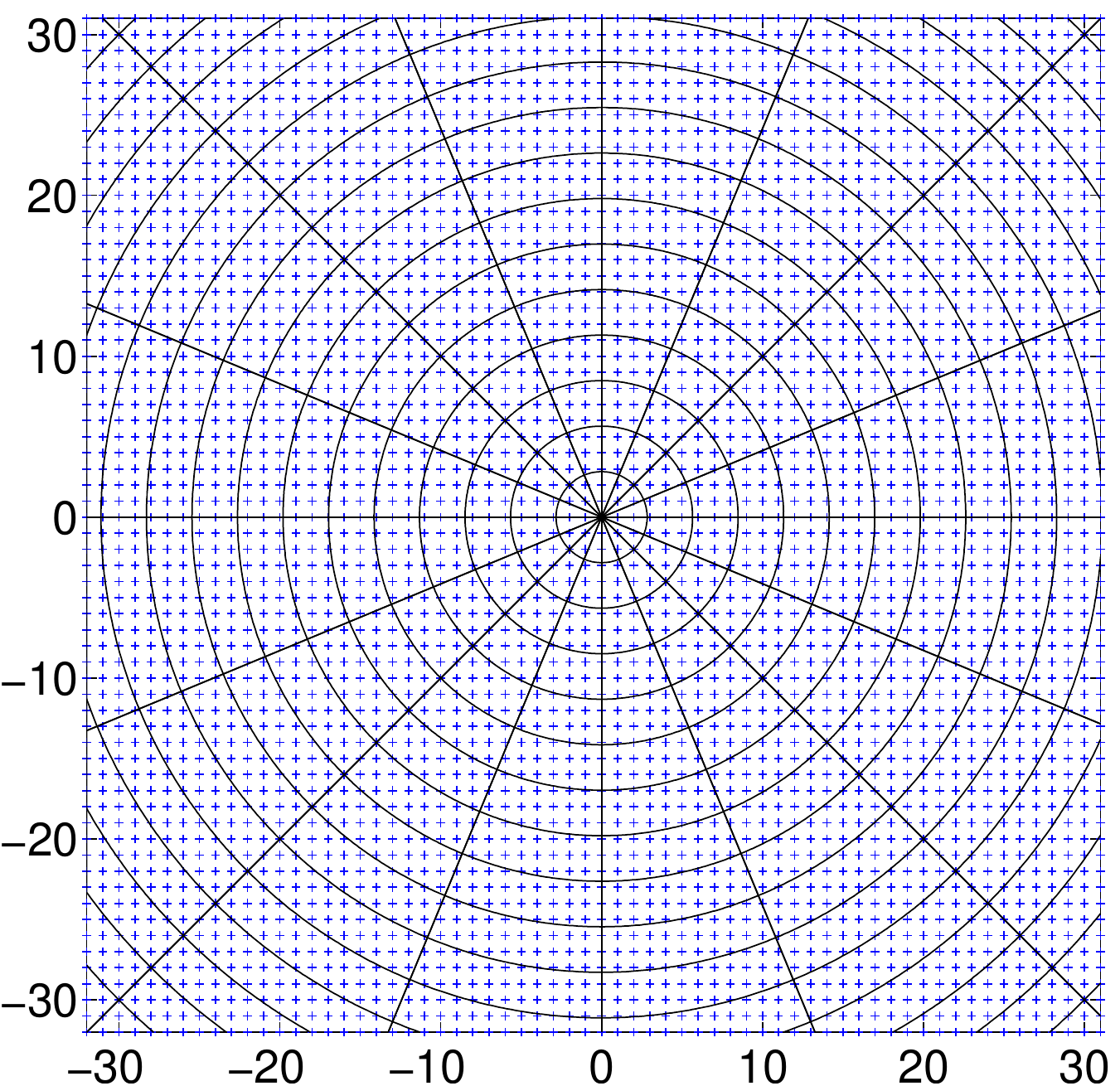} \\
      (a) & (b) & (c)
    \end{tabular}
  \end{center}
  \caption{ The point distribution and hierarchical partitioning (at a
    fixed level) for $N=64$.  (a) the set $X$.  (b) the set $P$ in
    polar coordinates.  (c) the frequency partitioning in Cartesian
    coordinates ($k \in \Omega$).  }
  \label{fig:ptree}
\end{figure}

With these notations, we can reformulate the computational problem
\eqref{eq:dfio} as
\[
u(x) = \sum_{p\in P} e^{2\pi\i N \Psi(x,p)} f(p), \quad x\in X,
\]
in which the sources $\{f(p)\}$ are now indexed by $p$ instead of $k$.
As we just mentioned, the main issue is that the kernel function
$e^{2\pi\i N \Psi(x,p)}$ is highly oscillatory. Our approach relies on
the observation that this kernel, properly restricted to time and
frequency subdomains, admits accurate and low-order separated
approximations.  To see why this is true, consider two square boxes
$A$ and $B$ in $[0,1]^2$ centered at $x_0(A)$ and $p_0(B)$, and suppose that
the sidelengths $w(A)$ and $w(B)$ obey the relationship $w(A)\,w(B) \le
1/N$. Introduce the new function
\begin{equation}
  R^{AB}(x,p) := \Psi(x,p)-\Psi(x_0(A),p)-\Psi(x,p_0(B)) + \Psi(x_0(A),p_0(B)),
  \label{eq:RAB}
\end{equation}
for each $x \in A$ and $p \in B$, and decompose the kernel $e^{2\pi\i
  N \Psi(x,p)}$ as
\begin{equation}
  e^{2\pi\i N \Psi(x,p)} = 
  e^{2\pi\i N \Psi(x_0(A),p)} \, e^{2\pi\i N \Psi(x,p_0(B))} \,
  e^{-2\pi\i N \Psi(x_0(A),p_0(B))} \,  e^{2\pi\i N R^{AB}(x,p)}.
  \label{eq:terms}
\end{equation}
In \eqref{eq:terms}, we note that each of the first three terms
depends on at most one variable ($x$ or $p$). Recall now the standard
multi-index notation; $i$ and $j$ are multi-indices and for
$i=(i_1,i_2)$, $i_1,i_2 \ge 0$, $|i| = i_1 + i_2$ and for
$x=(x_1,x_2)$, $x^i = x_1^{i_1} x_2^{i_2}$.  Applying the mean value
theorem to $R^{AB}(x,p)$ successively in $p$ and $x$ gives
\begin{eqnarray}
  \nonumber
  R^{AB}(x,p) & \le & \sup_{p^* \in B} \,\, \sum_{|j|=1} \left| \p^j_p [\Psi(x,p^*) - \Psi(x_0(A),p^*)] \right| \, |(p-p_0(B))^j|\\
  \nonumber
  & \le & \sup_{x^* \in A} \, \sup_{p^* \in B} \,\, \sum_{|i|=1} \sum_{|j|=1} \left| \p^i_x \p^j_p \Psi(x^*,p^*) \right| \, |(x-x_0(A))^i| \, |(p-p_0(B))^j|\\
  & = & O(1/N).
  \label{eq:R}
\end{eqnarray}
The last equation follows from the smoothness of $\Psi$ and from the
assumption $w(A)\,w(B)\le 1/N$. To summarize, \eqref{eq:R} gives $2\pi N
\, R^{AB}(x,p) = O(1)$ and, therefore, the complex exponential
$e^{2\pi\i N R^{AB}(x,p)}$ is nonoscillatory. 

Under some mild smoothness condition, this observation guarantees that
for any fixed accuracy $\eps$, there exists a low-rank separated
approximation of $e^{2\pi\i N R^{AB}(x,p)}$, valid over $A \times B$,
effectively decoupling the spatial variable $x$ from the frequency
variable $p$.  We propose constructing this low-rank approximation
using a tensor-product Chebyshev interpolation of the function
$e^{2\pi\i N R^{AB}(x,p)}$ in the $x$ variable when $w(A) \le
1/\sqrt{N}$, and in the $p$ variable when $w(B) \le 1/\sqrt{N}$. Since
the first three terms in \eqref{eq:terms} depend on at most one
variable, one also has a separated approximation of $e^{2\pi\i N
  \Psi(x,p)}$ with exactly the same separation rank. Looking at
\eqref{eq:terms}, the resulting low-rank approximation of the kernel
$e^{2\pi\i N \Psi(x,p)}$ can be viewed as a special interpolation
scheme that prefactors the oscillatory component, interpolates the
remaining nonoscillatory part, and finally appends the oscillatory
component. As we will see later, the separation rank providing an
$\eps$-approximation, for any fixed $\eps$, is bounded from above by a
constant independent of $N$. Further, if we define the potential
generated by the sources $p$ inside $B$ for any fixed box $B$ as
\begin{equation}
  u^B(x) = \sum_{p\in B} e^{2\pi\i N \Psi(x,p)} f(p),
  \label{eq:uB}
\end{equation}
then the existence of such a separated approximation implies the
existence of a compact expansion for the restriction of $u^B(x)$ to
$A$, $\{u^B(x), x\in A\}$, of the form 
\begin{equation}
  \label{eq:compact}
  u^B(x) \approx \sum_{1 \le j \le r} \sum_{p \in B} \alpha_j^{AB}(x) \beta_j^{AB}(p) f(p) = \sum_{1 \le j \le r} \delta_j^{AB} \alpha_j^{AB}(x), \qquad \delta_j^{AB} = \sum_{p \in B}  \beta_j^{AB}(p) f(p).
\end{equation}
In \eqref{eq:compact}, the number $r$ of expansion coefficients $
\delta_j^{AB}$ is independent of $N$ for a fixed relative error
$\eps$, as we will see later.

% EJC: The current status is that these paragraphs read well now. We
% repeat ourselves a little but this may serve a pedagogical
% purpose. We might want to revisit this later, however.
The problem is then to compute these compact expansions. This is where
the basic structure of the butterfly algorithm
\cite{michielssen-1996-mmda,oneil-2007-ncabft} is powerful. A brief
overview is as follows. We start by building two quadtrees $T_X$ and
$T_P$ (see Figure \ref{fig:ptree}(a) and (b)) respectively in the
spatial and frequency domains with leaf nodes at level $L=\log_2
N$. For each leaf node $B \in T_P$, we first construct the expansion
coefficients for the potential $\{u^B(x), x\in A\}$ where $A$ is the
root node of $T_X$. This can be done efficiently because $B$ is a very
small box.  Next, we go down in $T_X$ and up in $T_P$ simultaneously.
For each pair $(A,B)$ with $A$ at the $\ell$-th level of $T_X$ and $B$
at the $(L-\ell)$-th level of $T_P$, we construct expansion
coefficients for $\{u^B(x), x\in A\}$.  As shall see later, the key
point is that this is done by using the expansion coefficients which
have been already computed at the previous level.  Finally, we arrive
at level $\ell = L$, i.~e.~at the root node of $T_P$. There $u^B(x) =
u(x)$, and since one has available all the compact expansions
corresponding to all the leaf nodes $A$ of $T_X$, one holds an
approximation of the potential $u(x)$ for all $x\in X$.

%------------
\subsection{Applications}

The discrete equation \eqref{eq:dfio} naturally arises as a numerical
approximation of a continuous-time FIO of the general form
\begin{equation}
  \label{eq:cfio}
  u(x) = \int_{\R^2} a(x,k) e^{2\pi\i \Phi(x,k)} f(k) d k. 
\end{equation}
Note that in \eqref{eq:dfio}, the problem is simplified by setting the
amplitude $a(x,k)$ to 1. The reason for making this simpler is that in
most applications of interest, $a(x,k)$ is a much simpler object than
the term $e^{2\pi\i \Phi(x,k)}$. For instance, $a(x,k)$ often has a
low-rank separated approximation, which is valid in $\R^2 \times \R^2$
and yields a fast algorithm \cite{bao-1996-cpdo, candes-2007-fcfio}. Hence, setting
$a(x,k)=1$ retains the essential computational difficulty.

A significant instance of \eqref{eq:cfio} is the solution operator to
the wave equation
\[
u_{tt}(x,t) - c^2 \Delta u(x,t) = 0
\]
with constant coefficients and $x \in \R^2$. With initial conditions of the form $u(x,0)
= u_0(x)$ and $u_t(x,0) = 0$, say, the solution $u(x,t)$ at any time
$t > 0$ is given by
\[
u(x,t) = \frac{1}{2}
\left(
\int_{\R^2} e^{2\pi\i(x\cdot k + c|k|t)} \hat{u}_0(k) d k + 
\int_{\R^2} e^{2\pi\i(x\cdot k - c|k|t)} \hat{u}_0(k) d k
\right),
\]
where $\hat{u}_0$ is the Fourier transform of $u_0$.  Clearly, this is
the sum of two FIOs with phase functions $\Phi_{\pm}(x,k) = x\cdot k
\pm c|k|t$ and amplitudes $a_\pm (x,k) = 1/2$. Further, FIOs are still
solution operators even in the case of inhomogeneous coefficients
$c(x)$ as in
\[
u_{tt}(x,t) - c^2(x) \Delta u(x,t) = 0.
\]
Indeed, under very mild smoothness assumptions, the solution operator
remains the sum of two FIOs, at least for sufficiently small
times. The only difference is that the phases and amplitudes are a
little more complicated. In particular, the phase function is the
solution of a Hamilton-Jacobi equation which depends upon $c(x)$.

Another important example of FIO frequently arises in seismics.
A fundamental task in reflection seismology consists in
producing an image of the sharp features of an underground medium from
the seismograms generated by surface explosions. In a
nutshell, one builds an imaging operator which maps variations of the
pressure field at the surface into variations of the sound speed of
the medium (large variations indicate the presence of
reflectors). This imaging operator turns out to be an FIO
\cite{beylkin-1984-ipagrt,candes-2007-fcfio}.  Because FIOs are hard
to compute, several algorithms with various degrees of simplification have been proposed, 
most notably  {\em Kirchhoff migration} which approximates the imaging operator as a generalized Radon transform \cite{beylkin-1984-ipagrt, symes-1998-mfrs}. Computing this transform still has a
relatively high complexity, namely, of order $N^3$ in 2D. In contrast,
the algorithm proposed in this paper has an optimal $O(N^2 \log N)$
operation count, hence possibly offering a significant speedup.

%------------
\subsection{Related work}

Although FIOs play an important role in the analysis and computation
of linear hyperbolic problems, the literature on fast computations of
FIOs is surprisingly limited. The only work addressing
\eqref{eq:dfio} in this general form is the article
\cite{candes-2007-fcfio} by the authors of the current paper. The
operative feature in \cite{candes-2007-fcfio} is an angular
partitioning of the frequency domain into $\sqrt{N}$ wedges, each with
an opening angle equal to $2\pi/\sqrt{N}$. When restricting the input
to such a wedge, one can then factor the operator into a product of
two simpler operators. The first operator is provably approximately
low-rank (and lends itself to efficient computations) whereas the
second one is a nonuniform Fourier transform which can be computed
rapidly using the nonuniform fast Fourier transform (NFFT)
\cite{anderson-1996-rcdft,dutt-1993-fftnd,potts-2001-fftndt}.  The
resulting algorithm has an $O(N^{2.5}\log N)$ complexity.

In a different direction, there has been a great amount of research on
other types of oscillatory integral transforms. An important example
is the discrete $n$-body problem where one wants to evaluate sums of
the form 
\[
\sum_{1 \le j \le n} q_j K(|x-x_j|), \quad K(r) = e^{\i \omega r}/r
\]
in the high-frequency regime ($\omega$ is large).  Such problems
appear naturally when solving the Helmholtz equation by means of a
boundary integral formulation
\cite{colton-1998-iaest,colton-1983-iemst}.  A popular approach seeks
to compress the oscillatory integral operator by representing it in an
appropriate basis such as a local Fourier basis, or a basis extracted
from the wavelet packet dictionary
\cite{averbuch-2000-ecoi,bradie-1993-fnc,demanet-2008-sf,huybrechs-2006-twtmci}.
This representation sparsifies the operator, thus allowing fast
matrix-vector products. In spite of having good theoretical estimates,
this approach has thus far been practically limited to 1D
boundaries. One particular issue with this approach is that the
evaluation of the remaining nonnegligible coefficients sometimes
requires assembling the entire matrix, which can be computationally
rather expensive.

To the best of our knowledge, the most successful method for the
Helmholtz kernel $n$-body problem in both 2 and 3D is the
high-frequency fast multipole method (HF-FMM) proposed by Rokhlin and
his collaborators in a series of papers
\cite{rokhlin-1990-rsiest,rokhlin-1993-dfto,cheng-2006-awfmm}. This
approach combines the analytic property of the Helmholtz kernel with
an FFT-type fast algorithm to speedup the computation of the
interaction between well-separated regions. If $N^2$ is the number of
input and output points as before, the resulting algorithm has an
$O(N^2\log N)$ computational complexity. Other algorithms using
similar techniques can be found in
\cite{chew-2001-feace,darve-2000-fmmni,darve-2004-afmmmes,song-1995-mfma}.

Finally, the idea of butterfly computations has been applied to the
$n$-body problem in several ways. The original paper of Michielssen
and Boag \cite{michielssen-1996-mmda} used this technique to
accelerate the computation of the oscillatory interactions between
well-separated regions. More recently, Engquist and Ying
\cite{engquist-2007-fdmaok,engquist-2008-fdc} proposed a
multidirectional solution to this problem, where part of the algorithm
can be viewed as a butterfly computation between specially selected
spatial subdomain.

%------------
\subsection{Contents}

The rest of this paper is organized as follows. Section
\ref{sec:butterfly} describes the overall structure of the butterfly
algorithm. In Section \ref{sec:rank}, we prove the low-rank property
of the kernel and introduce an interpolation based method for
constructing low-rank separated approximations. Section \ref{sec:algo}
develops the algorithm by incorporating our low-rank approximations
into the butterfly structure. Numerical results are shown in Section
\ref{sec:results}. Finally, we discuss related problems for future
research in Section \ref{sec:concl}.

%----------------------------------------------------------
\section{The Butterfly Algorithm}
\label{sec:butterfly}

We begin by offering a general description of the butterfly structure
and then provide several concrete examples. This general structure was
originally introduced in \cite{michielssen-1996-mmda}, and later
generalized in \cite{oneil-2007-ncabft}.

In this section, $X$ and $P$ are two arbitrary point sets in $\R^d$,
both of cardinality $M$. We are given inputs $\{f(p), p \in P\}$ and
wish to compute the potentials $\{u(x), x\in X\}$ defined by
\[
u(x) = \sum_{p\in P} K(x,p) f(p), \quad x \in X, 
\]
where $K(x,p)$ is some kernel.  Let $D_X \supset X$ and $D_P \supset
P$ be two square domains containing $X$ and $P$ respectively. The main
data structure underlying the butterfly algorithm is a pair of dyadic
trees $T_X$ and $T_P$. The tree $T_X$ has $D_X$ as its root box and is
built by recursive, dyadic partitioning of $D_X$ until each leaf box
contains only a small number of points. The tree $T_P$ recursively
partitions $D_P$ in the same way. With the convention that the root
nodes are at level 0, one sees that under some uniformity condition
about the point distributions, the leaf nodes are at level $L = O(\log
M)$.  Throughout, $A$ and $B$ denote the square boxes of $T_X$ and
$T_P$, $\ell(A)$ and $\ell(B)$ denote their level.

The crucial property that makes the butterfly algorithm work is a
special low-rank property. Consider any pair of boxes $A \in T_X$ and
$B \in T_P$ obeying the condition $\ell(A) + \ell(B) = L$; we want the
submatrix $\{K(x,p), x \in A, p \in B\}$ (we will sometimes loosely
refer to this as the interaction between $A$ and $B$) to be
approximately of constant rank. More rigorously, for any $\eps$, there
must exist a constant $r_\eps$ independent of $M$ and two sets of
functions $\{\alpha^{AB}_t (x), 1 \le t \le r_\eps\}$ and
$\{\beta^{AB}_t (p), 1\le t \le r_\eps\}$ such that the following
approximation holds
\begin{equation}
  \left| K(x,p) - \sum_{t=1}^{r_\eps} \alpha^{AB}_t(x) \beta^{AB}_t(p) \right| \le \eps,
  \quad
  \forall x\in A,  \forall p\in B.
  \label{eq:glr}
\end{equation}
The number $r_\eps$ is called the {\em $\eps$-separation rank}. The
exact form of the functions $\{\alpha^{AB}_t (x) \}$ and
$\{\beta^{AB}_t (p)\}$ of course depends on the problem to which the
butterfly algorithm is applied, and we will give two examples at the
end of this section.

Recalling the definition $u^B(x) = \sum_{p\in B} K(x,p) f(p)$, the
low-rank property gives a compact expansion for $\{u^B(x), x\in A\}$
as summing \eqref{eq:glr} over $p\in B$ with weights $f(p)$ gives
\[
\left| u^B(x) - \sum_{t=1}^{r_\eps} \alpha^{AB}_t(x) \left( \sum_{p\in B} \beta^{AB}_t(p) f(p) \right) \right| \le \left( \sum_{p\in B} |f(p)| \right) \eps,
\quad \forall x \in A.
\]
Therefore, if we can find coefficients $\{\delta^{AB}_t\}_t$ obeying 
\begin{equation}
  \delta^{AB}_t \approx \sum_{p\in B} \beta^{AB}_t(p) f(p),
  \label{eq:delta}
\end{equation}
then the restricted potential $\{u^B(x), x\in A\}$ admits the compact
expansion
\[
\left| u^B(x) - \sum_{t=1}^{r_\eps} \alpha^{AB}_t(x) \delta^{AB}_t \right| \le \left( \sum_{p\in B} |f(p)| \right) \eps,
\quad \forall x\in A.
\]
We would like to emphasize that for each pair $(A,B)$, the number of
terms in the expansion is independent of $M$.

Computing $\{\delta^{AB}_t, 1\le t \le r_\eps\}$ by means of
\eqref{eq:delta} for all pairs $A, B$ is not efficient when $B$ is a
large box because for each $B$, there are many paired boxes $A$.  The
butterfly algorithm, however, comes with an efficient way for computing
$\{\delta^{AB}_t\}$ recursively.  The general structure of the
algorithm consists of a top down traversal of $T_X$ and a bottom up
traversal of $T_P$, carried out simultaneously. Postponing the issue
of computing the separated expansions, i.e.~$\{\alpha^{AB}_t(x)\}$ and
$\{\beta^{AB}_t(p)\}$, this is how the butterfly algorithm operates.
\begin{enumerate}
\item {\em Preliminaries}. Construct the trees $T_X$ and $T_P$ with
  root nodes $D_X$ and $D_P$.

\item {\em Initialization}. Let $A$ be the root of $T_X$. For each
  leaf box $B$ of $T_P$, construct the expansion coefficients $\{
  \delta^{AB}_t, 1\le t \le r_\eps\}$ for the potential $\{u^B(x),
  x\in A\}$ by simply setting
  \begin{equation}
    \delta^{AB}_t = \sum_{p\in B} \beta^{AB}_t(p) f(p).
  \label{eq:bf1}
  \end{equation}

\item {\em Recursion.} For $\ell = 1, 2, \ldots, L$, visit level
  $\ell$ in $T_X$ and level $L-\ell$ in $T_P$. For each pair $(A,B)$
  with $\ell(A) = \ell$ and $\ell(B) = L-\ell$, construct the expansion
  coefficients $\{\delta^{AB}_t, 1\le t \le r_\eps\}$ for the
  potential $\{u^B(x), x\in A\}$. This is done by using the low-rank
  representation constructed at the previous level ($\ell = 0$ is the
  initialization step). Let $A_p$ be $A$'s parent and $\{B_c\}$ be
  $B$'s children. At level $\ell-1$, the expansion coefficients
  $\{\delta^{A_p B_c}_{t'}\}_{t'}$ of $\{u^{B_c}(x), x\in A_p\}$ are
  readily available and we have
  \[
  \left| u^{B_c}(x) - \sum_{t'=1}^{r_\eps} \alpha^{A_pB_c}_{t'}(x) \delta^{A_pB_c}_{t'} \right| \le \left( \sum_{p\in B_c} |f(p)| \right) \eps,
  \quad \forall x\in A_p.
  \]
  Since $u^B(x) = \sum_c u^{B_c}(x)$, the previous inequality implies
  that
  \[
  \left| u^B(x) - \sum_c \sum_{t'=1}^{r_\eps} \alpha^{A_pB_c}_{t'}(x) \delta^{A_pB_c}_{t'} \right| \le \left( \sum_{p\in B} |f(p)| \right) \eps,
  \quad \forall x\in A_p.
  \]
  Since $A \subset A_p$, the above approximation is of course true for
  any $x \in A$. However, since $\ell(A) + \ell(B) = L$, the sequence of
  restricted potentials $\{u^B(x), x\in A\}$ also has a low-rank
  approximation of size $r_\eps$, namely,
  \[
  \left| u^B(x) - \sum_{t=1}^{r_\eps} \alpha^{AB}_t(x) \delta^{AB}_t \right| \le \left( \sum_{p\in B} |f(p)| \right) \eps,
  \quad \forall x\in A.
  \]
  Combining these last two approximations, we obtain that
  $\{\delta^{AB}_t\}_t$ should obey 
  \begin{equation}
    \sum_{t=1}^{r_\eps} \alpha^{AB}_t(x) \delta^{AB}_t \approx
    \sum_c \sum_{t'=1}^{r_\eps} \alpha^{A_pB_c}_{t'}(x) \delta^{A_pB_c}_{t'}, \quad \forall x\in A.
    \label{eq:bf2}
  \end{equation}
  Since this is an overdetermined linear system for
  $\{\delta^{AB}_t\}_t$ when $\{\delta^{A_pB_c}_{t'}\}_{t',c}$ are
  available, one possible approach to compute $\{\delta^{AB}_t\}_t$ is
  to solve a least-squares problem but this can be very costly when
  $|A|$ is large. Instead, the butterfly algorithm uses an approximate
  linear transformation mapping $\{\delta^{A_pB_c}_{t'}\}_{t',c}$ into
  $\{\delta^{AB}_t\}_t$, which can be computed efficiently.  We will
  discuss how this is done in several examples at the end of this
  section.

\item {\em Termination.} Now $\ell = L$ and set $B$ to be the root
  node of $T_P$. For each leaf box $A \in T_X$, use the constructed
  expansion coefficients $\{\delta^{AB}_t\}_t$ to evaluate $u(x)$ for
  each $x \in A$,
  \begin{equation}
    u(x) = \sum_{t=1}^{r_\eps} \alpha^{AB}_t (x) \delta^{AB}_t.
    \label{eq:bf3}
  \end{equation}
\end{enumerate}

A schematic illustration of the algorithm is provided in Figure
\ref{fig:algo}. We would like to emphasize that the strict balance
between the levels of the target boxes $A$ and source boxes $B$
maintained throughout the procedure is the key to obtaining accurate
low-rank separated approximations.

\begin{figure}[h!]
  \begin{center}
    \includegraphics[height=2in]{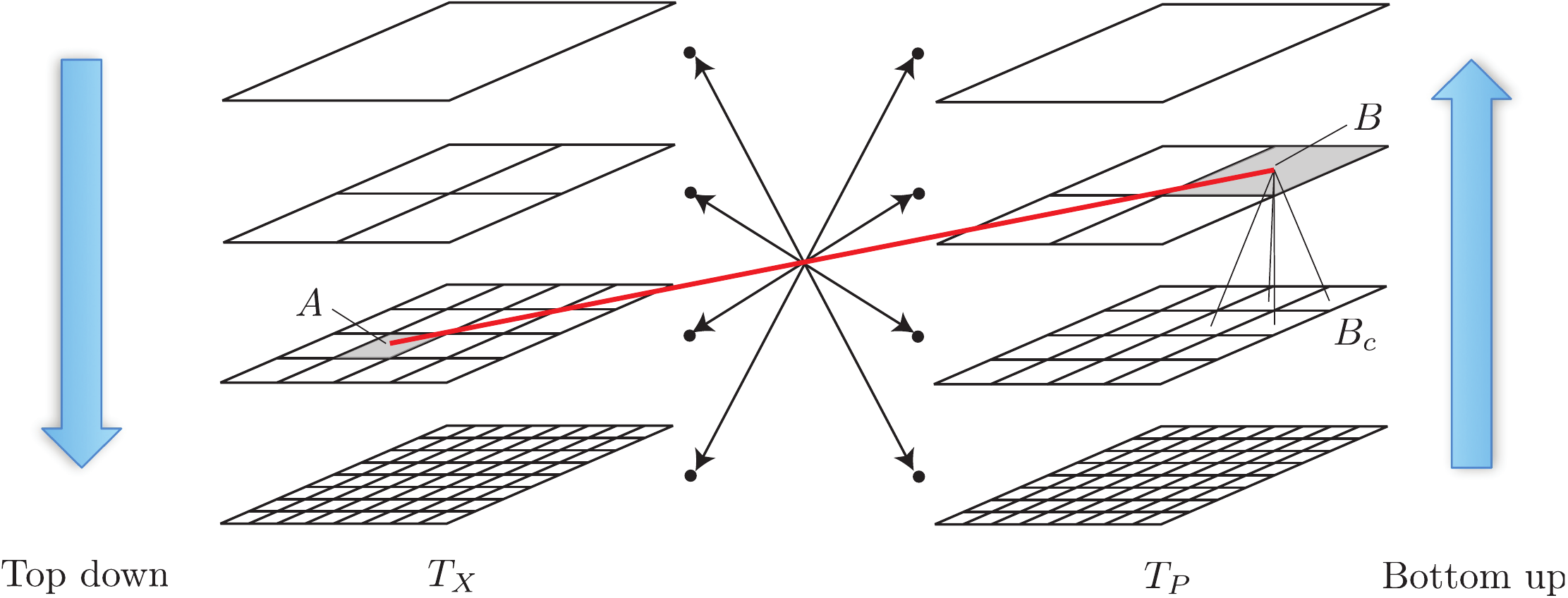}
  \end{center}
  \caption{Schematic illustration of the butterfly algorithm in 2D
    with 4 levels ($L = 3$).  The tree $T_X$ is on the left and $T_P$
    is on the right. The levels are paired as indicated so that the
    product of the sidelengths remains constant. The red line pairs
    two square boxes $A$ and $B$ at level 2 (shaded in gray); low-rank
    approximations of the localized kernel and expansion coefficients
    for the localized potential are computed for each such pair.  The
    algorithm starts at the root of $T_X$ and at the bottom of
    $T_P$. It then traverses $T_X$ top down and $T_P$ bottom up, and
    terminates when the last level (the bottom of $T_X$) is
    reached. The figure also represents the four children of any box
    $B$.  }
  \label{fig:algo}
\end{figure}

Leaving aside the computations of the separated expansion and taking
for granted that constructing $\{\delta^{AB}_t\}_t$ for each pair
$(A,B)$ has, in principle, the complexity of applying a linear
transform of size $O(r_\eps \times r_\eps)$, observe that the
butterfly algorithm has low computational complexity. To be sure, the
construction of $T_X$ and $T_P$ clearly takes at most $O(M\log M)$
operations. The initialization and termination steps take at most
$O(r_\eps \, M)$ as these steps require at most $O(r_\eps)$ operations
per point, see \eqref{eq:bf1} and \eqref{eq:bf3}.  The main workload
is of course in the recursion step. At each fixed level $\ell$, the
number of pairs $(A,B)$ under consideration is of order $O(M)$.  It
follows from our assumption that the number of flops required to
compute all the coefficients $\{\delta^{AB}_t\}_t$ at each level
$\ell$ is just $O(r_\eps^2 \, M)$.  Since there are only about $\log
M$ levels, the number of operations in the recursion is at most of the
order of $O(r_\eps^2 \, M\log M)$.  In conclusion, the overall
operation count is $O(r_\eps^2 \, M\log M)$.

The general structure of the butterfly algorithm should be clear by
now but we have left out two critical pieces, which we would need to
address to apply it to specific problems.
\begin{enumerate}
\item What are the functions $\{\alpha^{AB}_t(x)\}$ and
  $\{\beta^{AB}_t(p)\}$ in the low-rank approximation \eqref{eq:glr}
  and how are they computed?
\item How to solve for $\{\delta^{AB}_t\}_t$ from \eqref{eq:bf2}?
\end{enumerate}
The rest of this section discusses answers in two distinct examples.

\paragraph{Example 1.} In \cite{oneil-2007-ncabft}, O'Neil and Rokhlin
apply the butterfly algorithm to several special function transforms
in one dimension. Suppose that $N$ is a positive integer. In this
setup, $D_X = D_P = [0,N]$, $X$ and $P$ are two sets of $M = O(N)$
points distributed uniformly or quasi-uniformly in $[0,N]$, and the
kernel $K(x,p)$ parametrizes some special functions. For example, in
the case of the Fourier transform, $K(x,p) = e^{2\pi\i x p/N}$ so that
$p$ parametrizes a set of complex sinusoids. The trees $T_X$ and $T_P$
are recursive dyadic partitions of $[0,N]$ until the leaf nodes are of
unit size. In this work, all the kernels under study have low-rank
approximations when restricted to any pair $A\in T_X$ and $B\in T_P$
obeying $\ell(A) + \ell(B) = L = \log_2 N$.

The main tool for constructing the low-rank approximation is the
interpolative decomposition proposed in
\cite{gu-1996-eacsrqf,cheng-2005-clrm}. Given an $m\times n$ matrix
$Z$ which is approximately of rank $r$, the interpolative
decomposition constructs an approximate factorization $Z \approx Z_C
R$, where the matrix $Z_C$ consists of a subset of $r$ columns taken
from the original matrix $Z$ and the entries of $R$ have values close
to one.  Such a decomposition requires $O(m n^2)$ operations while
storing the matrix $R$ requires $O(r n)$ memory space.  Applying this
strategy to the kernel $K(x,p)$ with $x\in A$ and $p\in B$ implies
that the functions $\{ \alpha^{AB}_t(x), 1 \le t \le r\}$ are of the
form $\{K(x,p^{AB}_t), 1 \le t \le r \}$ with $\{p^{AB}_t\} \subset B$
and the functions $\{\beta^{AB}_t(p), 1 \le t \le r\}$ are given by
the corresponding entries in the matrix $R$.  Due to the special form
of $\{\alpha^{AB}_t(x)\}$, the coefficients $\{\delta^{AB}_t \}_t$ are
often called {\em equivalent sources}.

Now that we have addressed the computations of $\{\alpha^{AB}_t(x)\}$
and $\{\beta^{AB}_t(p)\}$, it remains to examine how to evaluate the
coefficients $\{\delta^{AB}_t\}$.  In the butterfly algorithm, these
coefficients are computed in the initialization step \eqref{eq:bf1}
and in the recursion step \eqref{eq:bf2}. Initially, $A$ is the root
box of $T_X$ and $B$ is a leaf box of $T_P$. To compute
$\{\delta^{AB}_t, 1 \le t \le r\}$ in the initialization step,
construct the interpolative decomposition for $K(x,p)$ with $x\in A$
and $p\in B$ obeying
\begin{equation}
  \left| K(x,p) - \sum_{t=1}^{r_\eps} K(x,p^{AB}_t) \beta^{AB}_t(p) \right| \le \eps, \quad \forall x\in A, \forall p\in B.
  \label{eq:pre1}
\end{equation}
Since each leaf box $B$ contains only a constant number of points $p$,
constructing the interpolative decomposition requires $O(N)$
operations and $O(r_\eps)$ memory space for each $B$. Since there at
most $O(N)$ of these boxes, the computational costs scales at most
like $O(N^2)$. Then we simply compute $\{\delta^{AB}_t, 1 \le t \le
r\}$ via \eqref{eq:bf1}. Once the interpolative decomposition is
available, this requires $O(r_\eps N)$ operation for all pairs at the
$0$th level.

% EJC: I know that we need to build a case for the interpolation
% method but this is a lot of details for something that we will
% actually not use. We have to think about whether everything is
% needed here. What I don't want is to confuse the reader who has to
% digest so much before getting to our stuff. We should avoid
% unecessary distractions.
%
% I suggest we keep it as is at the moment knowing that we may want to
% shorten it later on (e.g. during the revision).

As for \eqref{eq:bf2}, the special form of the functions
$\{\alpha^{AB}_t(x), 1 \le t \le r_\eps\}$ allows rewriting the
right-hand side as
\[
u^B(x) \approx \sum_c \sum_{t'=1}^{r_\eps} K(x,p^{A_pB_c}_{t'})
\delta^{A_pB_c}_{t'}.
\]
As a result, we can treat this quantity as the potential generated by the
equivalent sources $\{\delta^{A_pB_c}_{t'}\}_{c,t'}$ located at
$\{p^{A_pB_c}_{t'}\}_{c,t'}$. In order to find $\{\delta^{AB}_t, 1 \le
t \le r_\eps\}$, construct the interpolative decomposition of $K(x,p)$
with $x\in A$ and $p \in \{p^{A_pB_c}_{t'}\}_{c,t'}$, namely,
\begin{equation}
  \left| K(x,p) - \sum_{t=1}^{r_\eps} K(x,p^{AB}_t) \beta^{AB}_t(p) \right| \le \eps,
  \quad \forall x\in A, \forall p\in \{p^{A_pB_c}_{t'}\}_{c,t'}.
  \label{eq:pre2}
\end{equation}
Since the numbers of points in $\{p^{A_p B_c}_{t'}\}_{c,t'}$ is
proportional to $r_\eps$, this construction requires $O(r_\eps^2 \,
|A|)$ and requires $O(r_\eps^2)$ memory space per pair $(A,
B)$. Summing \eqref{eq:pre2} over $p\in \{p^{A_pB_c}_{t'}\}_{c,t'}$
with weights $\{\delta^{A_pB_c}_{t'}\}_{c,t'}$ gives a way to compute
$\{\delta^{AB}_t\}_t$. Indeed, one can set
\[
\delta^{AB}_t = \sum_c \sum_{t'} \beta^{AB}_t(p^{A_pB_c}_{t'}) \delta^{A_pB_c}_{t'},\quad
1\le t \le r_\eps.
\]
% EJC: Again, we would have to worry about the accumulation of
% errors. I.e., in 2D, \eps becomes 4 \eps then 16 \eps and so on. I
% will address this later. 

% LY: Please see the algorithm description in Section 2. The error is
% always bounded by $(\sum_{p\in B} |f(p)|)$. The absolute error does
% increase along with the size of $B$.

From the above discussion, we see that the butterfly algorithm
described in \cite{oneil-2007-ncabft} requires a precomputation step
to generate interpolative decompositions for
\begin{itemize}
\item $K(x,p)$ for $x\in A$ where $A$ is the root node and $p\in B$
  for each leaf node \eqref{eq:pre1},  
\item and $K(x,p)$ for $x\in A$ and $p\in \{p^{A_pB_c}_{t'}\}_{c,t'}$
  for each pair $(A,B)$ with $\ell(A) = 1, 2, \ldots, L$ and $\ell(A) +
  \ell(B) = L$ \eqref{eq:pre2}.
\end{itemize}
A simple analysis shows that these ``precomputations'' take $O(r_\eps^2
\, N^2)$ operations and require $O(r_\eps^2 \, N \log N)$ memory
space.  The quadratic time is very costly for problems with large
$N$. This might be acceptable if the same Fourier integral operator
were applied a large number of times. However, in the situation where
the operator is applied only a few times, the quadratic precomputation
step becomes a huge overhead, and the computational time may even
exceed that of the direct evaluation method. Moreover, the storage
requirement quickly becomes a bottleneck even for problems of moderate
sizes as in practice, the constant $r_\eps^2$ is often nonnegligible.

\paragraph{Example 2.} In \cite{ying-2008-sftba}, the butterfly
algorithm is used to develop a fast algorithm for sparse Fourier
transforms with both spatial and Fourier data supported on curves.
Suppose $N$ is a positive integer. In this setting, $D_X = D_P =
[0,N]^2$, and $X$ and $P$ are two set of $M = O(N)$ points supported
on smooth curves in $[0,N]^2$. The kernel is given by $K(x,p) =
e^{2\pi\i x\cdot p /N}$.  The quadtrees $T_X$ and $T_P$ are generated
adaptively in order to prune branches which do not intersect with the
support curves.  The leaf boxes are of unit size and $L=\log_2 N$. For
any pair of boxes $A\in T_X$ and $B\in T_P$ with $\ell(A) + \ell(B) =
L$, it is shown that the restricted kernel $K(x,p)$ is approximately
low-rank. Here, the functions $\{\alpha^{AB}_t(x), 1 \le t \le r\}$
take the form $\{ K(x,p^B_t), 1 \le t \le r\}$ where $\{p^B_t\}$ is a
tensor-product Chebyshev grid located inside the box $B$. The
coefficients $\{\delta^{AB}_t\}$---also called equivalent
sources---are constructed by collocating \eqref{eq:bf2} on a
tensor-product Chebyshev grid inside the box $A$.

Due to the tensor-product structure of the grid $\{p^B_t, 1 \le t \le
r\}$ and the special form of the kernel $K(x,p) = e^{2\pi\i x\cdot p
  /N}$, one can compute $\{\delta^{AB}_t\}$ via a linear
transformation which is essentially independent of the boxes $A$ and
$B$. As a result, one does not need the quadratic-time precomputation
step and there is no need to store explicitly these linear
transformations. We refer to \cite{ying-2008-sftba} for more details.

As we shall see, the algorithm introduced in this paper also makes use
of tensor-product Chebyshev grids, but the low-rank approximation is
constructed through interpolation rather than through
collocation. Before discussing other similarities and differences,
however, we first need to introduce our algorithm.

%----------------------------------------------------------
\section{Low-rank Approximations}
\label{sec:rank}

Recall that our problem is to compute 
\[
u(x) = \sum_{p\in P} e^{2\pi\i N \Psi(x,p)} f(p), \quad \Psi(x,p) =
\frac{\sqrt{2}}{2} \Phi(x,e^{2\pi\i p_2}) p_1,
\]
for all $x \in X$, where $X$ and $P$ are the point sets given in
Figure \ref{fig:ptree}(a) and (b). As both $X$ and $P$ are contained
in $[0,1]^2$, we set $D_X = [0,1]^2$ and likewise for $D_P$. Then the
two quadtrees $T_X$ and $T_P$ recursively partition the domains $D_X$
and $D_P$ uniformly until the finest boxes are of sidelength $1/N$.
%see Figure \ref{fig:ptree}.

%----------------------------
\subsection{The low-rank property}
\label{subsec:lrp}

We assume that the function $\Psi(x,p)$ is a real-analytic function in
the joint variables $x$ and $p$. This condition implies the existence
of two constants $Q$ and $R$ such that
\[
\sup_{x,p \, \in \, [0,1]^2} \, \left| \p^i_x \p^j_p \Psi(x,p) \right|
\le Q \, i!  j! \, R^{-|i|-|j|}, 
\]
where $i=(i_1,i_2)$ and $j=(j_1,j_2)$ are multi-indices, $i! = i_1!$
and $|i|= i_1 + i_2$.  For instance, the constant $R$ can be set as
any number smaller than the uniform convergent radius of the power
series of $\Psi$.  Following \cite{candes-2007-fcfio}, we term these
functions {\em $(Q,R)$-analytic}.

The theorem below states that for each pair of boxes $(A,B) \in T_X
\times T_P$ obeying $w(A)\,w(B) = 1/N$, the submatrix $\{e^{2\pi\i N
  R^{AB}(x,p)}, x \in A, p \in B\}$ is approximately
low-rank. Throughout, the notation $f \lesssim g$ means $f \le C g$
for some numerical constant $C$ independent of $N$ and $\eps$.
\begin{theorem}
  \label{thm:main}
  Let $A$ and $B$ be boxes in $T_X$ and $T_P$ obeying $w(A)\,w(B) = 1/N$.
  For any $\eps \le \eps_0$ and $N \ge N_0$, where $\eps_0$ and $N_0$
  are some constants, there exists an approximation obeying
  \[
  \left|
    e^{2\pi\i N R^{AB}(x,p)} - \sum_{t=1}^{r_\eps} \alpha^{AB}_t(x) \beta^{AB}_t(p)
  \right|
  \le \eps
  \]
  with $r_\eps \lesssim \log^4 (1/\eps)$. Moreover,
  \begin{itemize}
  \item when $w(B) \le 1/\sqrt{N}$, the functions
    $\{\beta^{AB}_t(p)\}_t$ can all be chosen as monomials in
    $(p-p_0(B))$ with a degree not exceeding a constant times $\log^2
    (1/\eps)$,
  \item and when $w(A) \le 1/\sqrt{N}$, the functions
    $\{\alpha^{AB}_t(x)\}_t$ can all be chosen as monomials in
    $(x-x_0(A))$ with a degree not exceeding a constant times $\log^2
    (1/\eps)$.
  \end{itemize}
\end{theorem}

The proof of Theorem \ref{thm:main} uses the following elementary
lemma (see \cite{candes-2007-fcfio} for a proof).
\begin{lemma}
  For each $z_0>0$ and $\eps>0$, set $s_\eps = \lceil \max(2e z_0,
  \log_2 (1/\eps)) \rceil$. Then 
  \[
  \left| e^{\i z} - \sum_{t=0}^{s_\eps-1} \frac{(\i z)^t}{t!} \right|
  \le \eps, \quad \forall |z| \le z_0. 
  \]
  \label{lem:exp}
\end{lemma}

\begin{proof}[Proof of Theorem \ref{thm:main}] Below, we will drop the
  dependence on $A$ and $B$ in $x_0(A)$ and $p_0(B)$ for $A$ and $B$
  are fixed boxes.  Since $w(A)\,w(B) = 1/N$, we either have $w(A) \le
  1/\sqrt{N}$ or $w(B) \le 1/\sqrt{N}$ or both. Suppose for instance
  that $w(B) \le 1 /\sqrt{N}$. Then
  \begin{align*}
    R^{AB}(x,p) & = \Psi(x,p) - \Psi(x_0,p) - \Psi(x,p_0) + \Psi(x_0,p_0)\\
    &  = \left[\Psi(x,p)-\Psi(x_0,p)\right] - \left[\Psi(x,p_0)-\Psi(x_0,p_0)\right]\\
    & = H_x(p) - H_x(p_0),
\end{align*}
where $H_x(p) := \Psi(x,p)-\Psi(x_0,p)$; the subscript indicates
that we see $H$ as a function of $p$ and think of $x$ as a parameter.
The function $R^{AB}(x,p)$ inherits the analyticity from $\Psi(x,p)$,
and its truncated Taylor expansion may be written as
\begin{equation}
    R^{AB}(x,p) =
    \sum_{1\le|i|<K} \frac{ \p^i_p H_x(p_0) }{i!} (p-p_0)^i
    + \sum_{|i|=K}   \frac{ \p^i_p H_x(p^*) }{i!} (p-p_0)^i, 
    \label{eq:RABexp}
\end{equation}
where $p^*$ is a point in the segment $[p_0, p]$. For each $i$ with
$|i| = K$, we have 
\[
\p^i_p H_x(p^*) = \sum_{|j|=1}   \p^j_x \p^i_p
    \Psi(x^*,p^*)(x-x_0)^j,
\]
for some point $x^*$ in $[x_0, x]$ and, therefore, it follows from the
$(Q,R)$-analycity property that
\[
    \left|   \frac{ \p^i_p H_x(p^*)}{i!} (p-p_0)^i \right| 
     \le  2 \, Q R^{-(K+1)} w(A) \,  (w(B))^K \le 
     2 Q R^{-2} \frac{1}{N} \left( \frac{w(B)}{R} \right)^{(K-1)}. 
\]
Since $w(B) \le 1/\sqrt{N}$, $1/\sqrt{N} \le R/2 \Rightarrow w(B)/R \le
1/2$ and, therefore, for $N$ sufficiently large, 
\[
\left| \frac{ \p^i_p H_x(p^*)}{i!} (p-p_0)^i \right| \le 2^{-(K-2)} \,
\frac{Q R^{-2}}{N}.
\]
Because there are at most $K+1$ terms with $|i| = K$, it follows that
\begin{equation}
  2\pi N \left|R^{AB}(x,p) - \sum_{1\le|i|<K} \frac{\p^i_p
      H_x(p_0)}{i!} (p-p_0)^i \right| \le \pi (K+1) \, Q
  R^{-2} \, 2^{-(K-3)}. 
  \label{eq:piK}
\end{equation}
Set
\begin{equation}
  K= C_0 \log(1/\eps). 
  \label{eq:Kval}
\end{equation}
Then if $C_0$ is a sufficiently large numerical constant, the right-hand
side of \eqref{eq:piK} is smaller than $\epsilon$ and, therefore, 
\begin{equation}
  2\pi N  \left| 
    R^{AB}(x,p) - A(x,p) \right| \le \eps, \qquad A(x,p) := 
  \sum_{1\le|i|<K} \frac{\p^i_p H_x(p_0)}{i!} (p-p_0)^i. 
  \label{eq:thmR}
\end{equation}

Noting that $\left| e^{\i a}-e^{\i b} \right| \le |a-b|$, we see that
in order to obtain an $\eps$-accurate separated approximation for
$e^{2\pi\i N R^{AB}(x,p)}$, we only need to construct one for
$e^{2\pi\i N A(x,p)}$.  Our plan is to invoke Lemma \ref{lem:exp}. To
do this, we need an estimate on $A(x,p)$.  When $K=1$, the estimate in
\eqref{eq:piK} provides a bound of $R^{AB}(x,p)$
\[
2\pi N | R^{AB}(x,p) | \le 8 \pi Q R^{-2}.
\]
Combining this estimate with \eqref{eq:thmR} yields 
\[
2\pi N |A(x,p)| \le 2\pi N \left| R^{AB}(x,p) \right| + \eps \le
8\pi Q R^{-2} + \eps.
\]
By taking $\eps$ small enough, we can assume
\[
2e \left( 8 \pi Q R^{-2} + \eps \right) \le \log_2(1/\eps), 
\]
and Lemma \ref{lem:exp} gives a $\log(1/\eps)$-term $\eps$-accurate
approximation
\[
\left| e^{2\pi\i N A(x,p)} - 
  \sum_{t=0}^{\log(1/\eps)-1} \frac{(2\pi N A(x,p))^t}{t!}
\right| \le \eps.
  \]
  Expanding $(2\pi\i A(x,p))^t$ for each $t$ gives a sum in which each
  term is a function of $x$ times a monomial $(p-p_0)^k$ of degree
  $|k| \lesssim \log^2(1/\eps)$.  Since there are at most
  $O(\log^4(1/\eps))$ different choices for the multi-index $k$ in the
  expanded formula, combining the terms with the same multi-index $k$
  yields an $O( \log^4(1/\eps) )$-term $2\eps$-accurate separated
  approximation for $e^{2\pi\i N R^{AB}(x,p)}$ with factors
  $\{\beta^{AB}_t(p)\}$ of the form $(p-p_0)^k$ as claimed. 
  
  We studied the case $w(B)\le 1/\sqrt{N}$ but the method is identical
  when $w(A) \le 1/\sqrt{N}$. Write
  \[
  R^{AB}(x,p) =\left[\Psi(x,p)-\Psi(x,p_0)\right] - \left[\Psi(x_0,p)-\Psi(x_0,p_0)\right]
  \]
  and follow the same procedure. The resulting approximation has
  $O(\log^4(1/\eps))$ terms, but the factors $\{\alpha^{AB}_t(x)\}$
  are now of the form $(x-x_0)^k$ with $|k| \lesssim \log^2(1/\eps)$.
\end{proof}

Theorem \ref{thm:main} shows that the $\eps$-rank of $e^{2\pi\i N
  R^{AB}(x,p)}$ is bounded by a constant multiple of $\log^4 (1/\eps)$
for a prescribed accuracy $\eps$. Since 
\[
\Psi(x,p) = \Psi(x,p_0) + \Psi(x_0,p) - \Psi(x_0,p_0) + R^{AB}(x,p),
\]
a direct consequence is that $\{e^{2\pi\i N \Psi(x,p)}, x \in A, p \in
B\}$ has a separated approximation of the same rank. A possible
approach to compute these approximations would be to use the
interpolative decomposition described in Example 1 of Section
\ref{sec:butterfly}. However, this method suffers from two main
drawbacks discussed in that section limiting its applicability to
relatively small problems. This is the reason why we propose below a
different and faster low-rank approximation method.

%----------------------------
\subsection{Interpolation gives good low-rank approximations}
\label{subsec:interp}

The proof of Theorem \ref{thm:main} shows that when $w(B) \le
1/\sqrt{N}$, the $p$-dependent factors in the low-rank approximation
of $e^{2\pi\i N R^{AB}(x,p)}$ are all monomials in $p$. Similarly,
when $w(A) \le 1/\sqrt{N}$, the $x$-dependent factors are monomials in
$x$.  This suggests that an alternative to obtain a low-rank separated
approximation is to use polynomial interpolation in $x$ when $w(A) \le
1/\sqrt{N}$, and in $p$ when $w(B) \le 1/\sqrt{N}$.

For a fixed integer $q$, the Chebyshev grid of order $q$ on
$[-1/2,1/2]$ is defined by
\[
\left\{ z_i = \frac{1}{2} \cos \left( \frac{i\pi}{q-1} \right) \right\}_{0\le i\le q-1}.
\]
We use this to define tensor-product grids adapted to an arbitrary
squared box with center $c$ and sidelength $w$ as   
\[
\{c + w (z_{i_1}, z_{i_2}), i_1,i_2=0,1,\ldots, q-1\}.
\]
Given a set of grid points $\{z_i \in \R, 0 \le i \le q-1\}$, we
will also consider the family of Lagrange interpolation polynomials
$L_i$ taking value 1 at $z_i$ and 0 at the other grid points
\[
L_i(z; \{z_i\}) = \prod_{0\le j \le q-1, j\not= i}\frac{z - z_j}{z_i - z_j}.
\]
For tensor-product grids $\{z_{1,i_1}\} \times \{z_{2,{i_2}}\}$, we define
the 2D interpolation polynomials as
\[
L_{i}(z, \{z_i\}) = 
L_{i_1}(z_1, \{z_{1,i_1}\}) \, L_{i_2}(z_2, \{z_{2,i_2}\}), \quad i = (i_1, i_2).  
\]

The theorem below shows that Lagrange interpolation provides efficient
low-rank approximations. In what follows, $L_t^B$ is the 2D Lagrange
interpolation polynomial on the Chebyshev grid adapted to the box $B$.

\begin{theorem}
  \label{thm:intp}
  Let $A$ and $B$ be as in Theorem \ref{thm:main}. Then for any $\eps
  \le \eps_0$ and $ N \ge N_0$ where $\eps_0$ and $N_0$ are the
  constants in Theorem \ref{thm:main}, there exists $q_\eps \lesssim
  \log^2 (1/\eps)$ such that
  \begin{itemize}
  \item when $w(B) \le 1/\sqrt{N}$, the Lagrange interpolation of
    $e^{2\pi\i N R^{AB}(x,p)}$ in $p$ on a $q_\eps \times q_\eps$
    Chebyshev grid $\{p^B_t\}$ adapted to $B$ obeys
    \begin{equation}
      \left| e^{2\pi\i N R^{AB}(x,p)} - \sum_t e^{2\pi\i N R^{AB}(x,p^B_t)} \, L^B_t(p) \right|
      \le \eps, \quad \forall x\in A, \forall p\in B,
      \label{eq:intp1}
    \end{equation}
  \item and when $w(A) \le 1/\sqrt{N}$, the Lagrange interpolation of
    $e^{2\pi\i N R^{AB}(x,p)}$ in $x$ on a $q_\eps \times q_\eps$
    Chebyshev grid $\{x^A_t\}$ adapted to $A$ obeys
    \begin{equation}
      \left| e^{2\pi\i N R^{AB}(x,p)} - \sum_t L^A_t(x) \, e^{2\pi\i N R^{AB}(x^A_t,p)} \right| 
      \le \eps, \quad \forall x\in A, \forall p\in B.
      \label{eq:intp2}
    \end{equation}
  \end{itemize}
  Both \eqref{eq:intp1} and \eqref{eq:intp2} provide a low-rank
  approximation with $r_\eps = q_\eps^2 \lesssim \log^4(1/\eps)$
  terms.
\end{theorem}

The proof of the theorem depends on the following lemma.
\begin{lemma}
  \label{lem:cheb}
  Let $f(y_1,y_2) \in C([0,1]^2)$ and $V_q$ be the space spanned by
  the monomials $y_1^{\alpha_1}y_2^{\alpha_2}$ with $0 \le \alpha_1,
  \alpha_2 <q$. The projection operator mapping $f$ into its Lagrange
  interpolant on the $q\times q$ tensor-product Chebyshev grid obeys 
  \[
  \| f - \Pi_q f\| \le (1+ C\log^2 q) \, \inf_{g\in V_q} \| f- g \|
  \]
  for some numerical constant $C$, where $\|f\| = \sup_{y \in [0,1]^2}
 \, |f(y)|$.
\end{lemma}
The proof of this lemma is a straightforward generalization of the one
dimensional case, which can be found in \cite{rivlin-1974-cp}.

\begin{proof}[Proof of Theorem \ref{thm:intp}]
  Suppose that $w(B) \le 1/\sqrt{N}$ and pick $q_\eps = K \log(1/\eps)$
  where $K = C_0 \log(1/\eps)$ is given by \eqref{eq:Kval} in the
  proof of Theorem \ref{thm:main}. We fix $x\in A$, and view
  $e^{2\pi\i N R^{AB}(x,p)}$ as a function of $p\in B$.  Applying
  Lemma \ref{lem:cheb} to $e^{2\pi\i N R^{AB}(x,\cdot)}$ gives
  \[
  \left\| e^{2\pi\i N R^{AB}(x,\cdot)} - \Pi_{q_\eps} e^{2\pi\i N
      R^{AB}(x,\cdot)}\right\| \le (1+C \log^2 q_\eps) \, 
  \inf_{g\in V_{q_\eps}} \left\| e^{2\pi\i N R^{AB}(x,\cdot)} -g
  \right\|.
  \]
  Theorem \ref{thm:main} states that the functions $\{\beta^{AB}_t
  (p)\}_t$ are all monomials of degree less than $q_\eps$. Therefore,
  for a fixed $x$, the low-rank approximation in that theorem belongs
  to $V_{q_\eps}$, and approximates $e^{2\pi\i N R^{AB}(x,\cdot)}$ within
  $\eps$.  Combining this with the previous estimate gives
  \begin{equation}
    \left\| e^{2\pi\i N R^{AB}(x,\cdot)} - \Pi_{q_\eps} e^{2\pi\i N R^{AB}(x,\cdot)}\right\|
    \le (1+C \log^2 q_\eps) \, \eps
    \le (C_1 + C_2 \log^2 (\log (1/\eps))) \, \eps, 
    \label{eq:intp1aux}
  \end{equation}
  where $C_1$ and $C_2$ are two constants independent of $N$ and
  $\eps$. The same analysis applies to the situation where $w(A) \le
  1/\sqrt{N}$; fix $p\in B$ and view $e^{2\pi\i N R^{AB}(\cdot,p)}$ as
  a function of $x \in A$, repeat the same procedure and
  obtain the same error bound.  

  The estimate \eqref{eq:intp1aux} and its analog when $w(A) \le
  1/\sqrt{N}$ are the claims \eqref{eq:intp1} and \eqref{eq:intp2} but
  for the fact that the right-hand side is of the form $(C_1 + C_2
  \log^2 (\log (1/\eps))) \, \eps$ rather than $\eps$. In order to get
  rid of the $C_1 + C_2 \log^2 (\log (1/\eps))$ factor, we can
  repeat the proof with $\eps^{(1+\delta)}$ with a small
  $\delta>0$. As $q_\eps$ only depends on $\eps$ logarithmically, this
  only increases $q_\eps$ by a small constant factor.
\end{proof}

Finally, to obtain a low-rank approximation for the real kernel
$e^{2\pi\i N \Psi(x,p)}$ when $w(B) \le 1/\sqrt{N}$, multiply
\eqref{eq:intp1} with $e^{2\pi\i N \Psi(x_0,p)} \, e^{2\pi\i N
  \Psi(x,p_0)} \, e^{-2\pi\i N \Psi(x_0,p_0)}$ (we use again $x_0$ and
$p_0$ as shorthands for $x_0(A)$ and $p_0(B)$) which gives
\[
\left|
  e^{2\pi\i N \Psi(x,p)} - \sum_t e^{2\pi\i N \Psi(x,p^B_t)} 
  \left(
    e^{-2\pi\i N \Psi(x_0,p^B_t)} \, L^B_t(p) \, e^{2\pi\i N \Psi(x_0,p)}
  \right)
\right| \le \eps,\quad
\forall x\in A, \forall p\in B.
\]
In terms of the notations in \eqref{eq:glr}, the expansion functions
are given by
\begin{equation}
  \alpha^{AB}_t(x) = e^{2\pi\i N \Psi(x,p^B_t)},\quad
  \beta^{AB}_t(p) = e^{-2\pi\i N \Psi(x_0,p^B_t)} \, L^B_t(p) \, e^{2\pi\i N \Psi(x_0,p)}.
\label{eq:ab1}
\end{equation}
This is a special interpolant of the function $e^{2\pi\i N \Psi(x,p)}$ in
the $p$-variable which 1) prefactors the oscillation 2) performs the
interpolation and 3) remodulates the outcome. Following
\eqref{eq:delta}, the expansion coefficients $\{\delta^{AB}_t\}_t$ for
the potential $\{u^B(x), x\in A\}$ should then obey the condition
\begin{equation}
  \delta^{AB}_t \approx 
  \sum_{p\in B} \beta^{AB}_t(p) f(p) =
  e^{-2\pi\i N \Psi(x_0,p^B_t)} \sum_{p\in B} \left( L^B_t(p) \, e^{2\pi\i N \Psi(x_0, p)} \, f(p) \right).
  \label{eq:D1}
\end{equation}
When $w(A)\le 1/\sqrt{N}$, multiply \eqref{eq:intp2} with $e^{2\pi\i
  N \Psi(x_0,p)} \, e^{2\pi\i N \Psi(x,p_0)} \, e^{-2\pi\i N \Psi(x_0,p_0)}$
and obtain
\[
\left|
  e^{2\pi\i N \Psi(x,p)} - \sum_t
  \left(
  e^{2\pi\i N \Psi(x,p_0)} \, L^A_t(x) \, e^{-2\pi\i N \Psi(x^A_t,p_0)}
  \right)
  e^{2\pi\i N \Psi(x^A_t,p)}
\right|
\le \eps,\quad
\forall x\in A, \forall p\in B.
\]
The expansion functions are now
\begin{equation}
  \alpha^{AB}_t(x) = e^{2\pi\i N \Psi(x,p_0)} \, L^A_t(x) \, e^{-2\pi\i N \Psi(x^A_t,p_0)},\quad
  \beta^{AB}_t(p) = e^{2\pi\i N \Psi(x^A_t,p)}.
  \label{eq:ab2}
\end{equation}
The expansion coefficients $\{\delta^{AB}_t\}$ should obey
\begin{equation}
  \delta^{AB}_t \approx
  \sum_{p\in B} \beta^{AB}_t(p) f(p) =
  \sum_{p\in B} e^{2\pi\i N \Psi(x^A_t,p)} f(p) = u^B(x^A_t).
  \label{eq:D2}
\end{equation}

% ----
%------------------------------
\section{Algorithm Description}
\label{sec:algo}

This section presents our algorithm which combines the expansions
introduced in Section \ref{sec:rank} with the butterfly structure from
Section \ref{sec:butterfly}.
\begin{enumerate}
\item {\em Preliminaries}. Construct two quadtrees $T_X$ and $T_P$ for
  $X$ and $P$ as in Figure \ref{fig:ptree}. Each leaf node of $T_X$
  and $T_P$ is of size $1/N \times 1/N$. Since $X$ is a regular
  Cartesian grid, $T_X$ is just a uniform hierarchical partition.

\item {\em Initialization}. Set $A$ to be the root of $T_X$. For each
  leaf box $B \in T_P$, construct the expansion coefficients
  $\{\delta^{AB}_t, 1 \le t \le r_\eps\}$ from \eqref{eq:D1} by
  setting
  \begin{equation}
    \delta^{AB}_t = e^{-2\pi\i N \Psi(x_0(A),p^B_t)} \sum_{p\in B}
    \left( L^B_t(p) \, e^{2\pi\i N \Psi(x_0(A),p)} \, f(p) \right).
    \label{eq:alg2a}
\end{equation}

\item {\em Recursion}. For each $\ell = 1, 2, \ldots,L/2$, construct
  the coefficients $\{\delta^{AB}_t, 1 \le t \le r_\eps\}$ for each
  pair $(A,B)$ with $A$ at level $\ell$ and $B$ at the complementary
  level $L-\ell$ as follows: let $A_p$ be $A$'s parent and $\{B_c, c =
  1, 2, 3, 4\}$ be $B$'s children. For each child, we have available
  from the previous level an approximation of the form
  \[
  u^{B_c}(x) \approx \sum_{t'} e^{2\pi\i N \Psi(x,p^{B_c}_{t'})} \delta^{A_pB_c}_{t'},
  \quad \forall x\in A_p. 
  \]
  Summing over all children gives
  \[
  u^B(x) \approx \sum_c \sum_{t'} e^{2\pi\i N \Psi(x,p^{B_c}_{t'})} \delta^{A_pB_c}_{t'},
  \quad \forall x\in A_p.
  \]
  Since $A \subset A_p$, this is also true for any $x\in A$.  This
  means that we can treat $\{\delta^{A_pB_c}_{t'}\}$ as equivalent
  sources in $B$.  As explained below, we then set the coefficients
  $\{\delta^{AB}_t\}_t $ as
  \begin{equation}
    \delta^{AB}_t = e^{-2\pi\i N \Psi(x_0(A),p^B_t)} \sum_c \sum_{t'}
    L^B_t(p^{B_c}_{t'}) \, e^{2\pi\i N \Psi(x_0(A), p^{B_c}_{t'})} \,
    \delta^{A_pB_c}_{t'}.
    \label{eq:alg2b}
  \end{equation}

\item {\em Switch}. The interpolant in $p$ may be used as the
  low-rank approximation as long as $\ell \le L/2$ whereas the
  interpolant in $x$ is a valid low-rank approximation as soon as
  $\ell \ge L/2$. Therefore, at $\ell = L/2$, we need to switch
  representation.  Recall that for $\ell \le L/2$ the expansion
  coefficients $\{ \delta^{AB}_t, 1 \le t \le r_\eps\}$ may be
  regarded as equivalent sources while for $\ell \ge L/2$, they
  approximate the values of the potential $u^B(x)$ on the Chebyshev
  grid $\{x_t^A, 1 \le t \le r_\eps\}$. Hence, for any pair $(A,B)$
  with $A$ at level $L/2$ (and likewise for $B$), we have
  $\delta_t^{AB} \approx u^B(x_t^A)$ from \eqref{eq:D2} so that we
  may set
  \begin{equation}
    \delta^{AB}_t = \sum_s e^{2\pi\i N \Psi(x^A_t,p^B_s)}\, 
    \delta^{AB}_s
    \label{eq:alg2m}
  \end{equation}
  (we abuse notations here since $\{\delta_t^{AB}\}$ denotes the new
  set of coefficients and $\{\delta_s^{AB}\}$ the older set).
  
\item {\em Recursion (end)}. The rest of the recursion is analogous.
  For $\ell = L/2+1, \ldots, L$, construct the coefficients $\{
  \delta^{AB}_t, 1 \le t \le r_\eps\}$ as follows.  With
  $\{\alpha_t^{AB}\}$ and $\{\beta_t^{AB}\}$ given by \eqref{eq:ab2},
  we have 
  \[
  % \begin{align*}
  u^B(x) = \sum_c u^{B_c}(x)  \approx \sum_{t', c} \alpha_{t'}^{A_p
    B_c}(x) \sum_{p \in B_c} \beta_{t'}^{A_p B_c}(p) f(p)
  \approx \sum_{t', c} \alpha_{t'}^{A_p
    B_c}(x) \delta_{t'}^{A_p B_c}. 
  % \end{align*}
  \]
  Hence, since $\delta^{AB}_t$ should approximate $u^B(x^A_t)$ by
  \eqref{eq:D2}, we simply set
  \[
  \delta_t^{AB} = \sum_{t', c} \alpha_{t'}^{A_p B_c}(x) \delta_{t'}^{A_p
    B_c}.
  \] 
  Substituing $\alpha_t^{AB}$ with its value gives the update
  \begin{equation}
    \delta^{AB}_t = \sum_c e^{2\pi\i N \Psi(x^A_t, p_0(B_c))} \sum_{t'} 
    \left(
      L^{A_p}_{t'}(x^A_t) \, e^{-2\pi\i N \Psi(x^{A_p}_{t'},p_0(B_c))} \, \delta^{A_pB_c}_{t'}
    \right). 
    \label{eq:alg2c}
  \end{equation}

\item {\em Termination}. Finally, we reach $\ell = L$ and set $B$ to
  be the root box of $T_P$. For each leaf box $A$ of $T_X$, we have 
  \[
  u^B(x) \approx \sum_t \alpha_t^{AB}(x) \delta_t^{AB}, \quad x \in A, 
  \]
  where $\{\alpha_t^{AB}\}$ is given by \eqref{eq:ab2}. Hence, for each
  $x \in A$, we set
  \begin{equation}
    u(x) = e^{2\pi\i N \Psi(x,p_0(B))} \sum_t 
    \left(
      L^{A}_t(x) \, e^{-2\pi\i N \Psi(x^A_t,p_0(B))} \, \delta^{AB}_t
    \right). 
    \label{eq:alg2d}
  \end{equation}
\end{enumerate}

In order to justify \eqref{eq:alg2b}, recall that
\[
\left|
  e^{2\pi\i N \Psi(x,p)} - \sum_t e^{2\pi\i N \Psi(x,p^B_t)} \beta_t^{AB}(p)
\right| \le \eps,\quad
\forall x\in A, \forall p\in B,
\] 
where $\beta_t^{AB}(p)$ is given by \eqref{eq:ab1}. Summing the above
inequality over $p \in \{p_{t'}^{B_c}\}_{t', c}$ with weights
$\{\delta_{t'}^{A_p B_c}\}$ gives
\[ u^B(x) \approx \sum_{t} e^{2\pi\i N \Psi(x,p^B_t)} \sum_{c, t'}
\beta_t^{AB}(p_{t'}^{B_c}) \delta_{t'}^{A_p B_c},  
\]
which means that we can set 
\[
\delta_t^{AB} = \sum_{c, t'}
\beta_t^{AB}(p_{t'}^{B_c}) \delta_{t'}^{A_p B_c}.  
\] 
Substituing $\beta_t^{AB}$ with its value gives the update
\eqref{eq:alg2b}.

% \begin{figure}[h!]
%   \begin{center}
%     \includegraphics[height=2in]{butterfly.pdf}
%   \end{center}
%   \caption{Schematic illustration of the butterfly algorithm
%     with 4 levels ($L = 3$).  The tree $T_X$ is on the left and $T_P$
%     is on the right. The levels are paired as indicated so that the
%     product of the sidelengths remains constant. The red line pairs
%     two square boxes $A$ and $B$ at level 2 (shaded in gray); low-rank
%     approximations of the localized kernel and expansion coefficients
%     for the localized potential are computed for each such pair.  The
%     algorithm starts at the root of $T_X$ and at the bottom of
%     $T_P$. It then traverses $T_X$ top down and $T_P$ bottom up, and
%     terminates when the last level (the bottom of $T_X$) is
%     reached. The figure also represents the four children of any box
%     $B$.
%   }
%   \label{fig:algo}
% \end{figure}

The main workload is in \eqref{eq:alg2b} and \eqref{eq:alg2c}.
Because of the tensor product structures, the computations in
\eqref{eq:alg2b} and \eqref{eq:alg2c} can be accelerated by performing
Chebyshev interpolation one dimension at a time, reducing the number
operations from $O(r_\eps^2) = O(q_\eps^4)$ to $O(q_\eps^3)$.  As
there are at most $O(N^2 \log N)$ pairs of boxes $(A,B)$, the
recursion steps take at most $O(r_\eps^{3/2} \, N^2\log N)$
operations. It is not difficult to see that the remaining steps of the
algorithm take at most $O(r_\eps^2 \, N)$ operations.  Hence, with
$r_\eps = O(\log^4(1/\eps))$, this gives an overall complexity
estimate of $O(\log^6(1/\eps) \, N^2 \log N + \log^8(1/\eps) \,
N^2)$. Since the prescribed accuracy $\eps$ is a constant, our
algorithm runs in $O(N^2 \log N)$ time with a constant polylogarithmic
in $\eps$.  Although the dependence of this constant on $\log(1/\eps)$
is quite strong, we would like to emphasize that this is only a worst
case estimate. In practice, and as empirically demonstrated in Section
\ref{sec:results}, this dependence is rather moderate and grows like
$\log(1/\eps)$.

We would like to point out that the values of $L^B_t(p^{B_c}_{t'})$ in
\eqref{eq:alg2b} and of $L^{A_p}_{t'}(x^A_t)$ in \eqref{eq:alg2c} are
both translation and level-independent because of the nested structure
of the quadtree. Therefore, once these values are computed for a
single pair $(A,B)$, they can just be reused for all pairs visited
during the execution of the algorithm. In our implementation, the
values of $L^B_t(p^{B_c}_{t'})$ in \eqref{eq:alg2b} and
$L^{A_p}_{t'}(x^A_t)$ are stored in a Kronecker-product form in order
to facilitate the dimension-wise Chebyshev interpolation discussed in
the previous paragraph.

This algorithm has two main advantages over the approach based on
interpolative decomposition. First, no precomputation is
required. Since the low-rank approximation uses Lagrange interpolation
on fixed tensor-product Chebyshev grids, the functions
$\{\alpha^{AB}_t(x)\}$ and $\{\beta^{AB}_t(p)\}$ are given explicitly
by \eqref{eq:ab1} and \eqref{eq:ab2}.  In turn, this yields explicit
formulas for computing the expansion coefficients
$\{\delta^{AB}_t\}_t$, compare \eqref{eq:alg2a}, \eqref{eq:alg2b}, and
\eqref{eq:alg2c}.  Second, this algorithm is highly efficient in terms
of memory requirement. In the approach based on the interpolative
decomposition method, one needs to store many linear transformations
(one for each pair $(A, B)$) which yields a storage requirement on the
order of $r_\eps^2 \, N^2 \log N$ as observed earlier. The proposed
algorithm, however, only needs to store the expansion coefficients $\{
\delta^{AB}_t\}$.  Moreover, at any point in the execution, only the
expansion coefficients from two consecutive levels are actually
needed. Therefore, the storage requirement is only on the order of
$r_\eps \, N^2$, which allows us to address problems with much larger
sizes.

One advantage of the interpolative decomposition approach is that it
often has a smaller separation rank. The reason is that the low-rank
approximation is optimized for the kernel under study and, therefore,
the computed rank is usually very close to the true separation rank
$r_\eps$. In contrast, our low-rank approximations are based on
tensor-product Chebyshev grids and merely exploit the smoothness of
the function $e^{2\pi\i N R^{AB}(x,p)}$ either in $x$ or in $p$. In
particular, it ignores the finer structure of the kernel $e^{2\pi\i N
  \Psi(x,p)}$ and as a result, the computed separation rank is often
significantly higher. Fortunately, this growth in the separation rank
does not result in a significant increase in the computation time
since the tensor-product structure and the Lagrangian interpolants
dramatically decrease the computational cost. 

The tensor-product Chebyshev grid is also used in the method described
in Example 2 of Section \ref{sec:butterfly}. There, the equivalent
sources are supported on a Chebyshev grid in $B$ and are constructed
by collocating the potential on another Chebyshev grid in $A$. Because
of 1) the tensor-product nature of the grids and 2) the nature of the
Fourier kernel, the matrix representation of this collocation
procedure has an almost $(A,B)$-independent Kronecker product
decomposition. This offers a way of speeding up the computations and
makes it unnecessary to store the matrix representation.
Unfortunately, such an approach would not work for FIOs since the
kernel $e^{2\pi\i N \Psi(x,p)}$ does not have an $(A,B)$-independent
tensor-product decomposition. This is why a major difference is that
we use tensor-product Chebyshev grids only to interpolate the residual
kernel $e^{2\pi\i N R^{AB}(x,p)}$ in $x$ or $p$ depending on which box
is smaller. The important point is that we also keep the main benefits
of that approach.

Up to this point, we have only been concerned with the computation of
FIOs with constant amplitudes. However, our approach can easily be
extended to the general case with variable amplitudes $a(x,k)$ as in 
\begin{equation}
  u(x) = \sum_{k\in \Omega} a(x,k) e^{2\pi\i \Phi(x,k)} f(k), \quad  x\in X.
  \label{eq:allfio}
\end{equation}
In most applications of interest, $a(x,k)$ is a simple object, i.e.~much
simpler than the oscillatory term $e^{2\pi\i \Phi(x,k)}$. A possible
approach is to follow \cite{bao-1996-cpdo,demanet-2008-dsc} where the
amplitude is assumed to have a low-rank separated approximation obeying 
\[
\left| a(x,k) - \sum_{t=1}^{s_\eps} g_t(x) h_t(k)\right| \le \eps,
\]
where the number of terms $s_\eps$ is independent of $N$---the size of
the grids $X$ and $\Omega$. Such an approximation can be obtained
either analytically or through the randomized procedure described in
\cite{candes-2007-fcfio}. An algorithm for computing \eqref{eq:allfio}
may then operate as follows: 
\begin{enumerate}
\item Construct the approximation $a(x,k) \approx \sum_{t=1}^{s_\eps}
  g_t(x) h_t(k)$ with $x\in X$ and $k\in \Omega$.
\item Set $u(x) = 0$ for $x \in X$ and for each $t=1,\ldots,s_\eps$,
  \begin{enumerate}
  \item form the product $f_t(k) = h_t(k) f(k)$ for $k\in \Omega$, 
  \item compute $\sum_k e^{2\pi\i \Phi(x,k)} f_t(k)$ for $x\in
    X$ by applying the above algorithm, 
  \item multiply the result with $g_t(x)$ for $x\in X$, and add this
    product to $u(x)$.
  \end{enumerate}
\end{enumerate}

We would like to point out that the above algorithm is presented in a
form that is conceptually simple. However, when applying the butterfly
algorithm to the functions $\{f_t(k), t=1,\ldots,s_\eps \}$ in the
multiple executions of Step 2(b), the following kernel evaluations are
independent of $\{f_t(k)\}$ and thus performed redundantly:
$e^{-2\pi\i N \Psi(x_0(A),p^B_t)}$ and $e^{2\pi\i N \Psi(x_0(A),
  p^{B_c}_{t'})}$ in \eqref{eq:alg2b}; $e^{2\pi\i N
  \Psi(x^A_t,p^B_s)}$ in \eqref{eq:alg2m}; $e^{2\pi\i N \Psi(x^A_t,
  p_0(B_c))}$ and $e^{-2\pi\i N \Psi(x^{A_p}_{t'},p_0(B_c))}$ in
\eqref{eq:alg2c}. Therefore, in an efficient implementation of the
above algorithm, one should ``vectorize'' the butterfly algorithm to
operate on $\{f_t(k), t=1,\ldots,s_\eps\}$ simultaneously so that
redundant kernel evaluations can be avoided.

% ----------------------------------------------------------
\section{Numerical Results}
\label{sec:results}

This section provides some numerical results to illustrate the
empirical properties of the algorithm. The implementation
is in C++ and all tests are carried out on a desktop computer with a
2.8GHz CPU. 

When computing the matrix-vector product $u(x) = \sum_{k\in \Omega}
e^{2\pi\i \Phi(x,k)} f(k)$, we independently sample the entries of
the input vector $\{f(k), k\in \Omega\}$ from the standard normal
distribution so that the input vector is just white noise. Let
$\{u^a(x), x\in X\}$ be the potentials computed by the algorithm.  To
report on the accuracy, we select a set $S$ of 256 points from $X$ and
estimate the relative error by
\begin{equation}
  \sqrt{ \frac{ \sum_{x\in S} |u(x) - u^a(x) |^2 } { \sum_{x\in S}
      |u(x)|^2 } }.
  \label{eq:err}
\end{equation}

According to the algorithm description in Section \ref{sec:algo}, the
leafs of the quadtree at level $L=\log_2 N$ are of size $1/N \times
1/N$ and each contains a small number of points. However, when the
number of points in a box $B$ is much less than $q_\eps^2$, it does
not make sense to construct the expansion coefficients
$\{\delta^{AB}_t\}$ simply because the sources at these points would
themselves provide a more compact representation. Thus in practice,
the recursion starts from the boxes in $T_P$ that are a couple of
levels away from the bottom so that each box has at least $q_\eps^2$
points in it. Similarly, the recursion stops at the boxes in $T_X$
that are a couple of levels away from the bottom. In general, the
starting and ending levels should depend on the value of $q_\eps$. In
the following examples, we start from level $\log_2N - 3$ and stop at
level $3$ in $T_P$. This choice matches well with the values of
$q_\eps$ (5 to 11) that we use here.

In our first example, we consider the computation of \eqref{eq:dfio}
with the phase function given by
\begin{equation}
  \Phi(x,k) = x \cdot k + \sqrt{c_1^2(x) k_1^2 + c_2^2(x) k_2^2}, \qquad
  \begin{array}{l}
    c_1(x) = (2 + \sin(2\pi x_1) \sin(2\pi x_2))/3, \\
    c_2(x) = (2 + \cos(2\pi x_1) \cos(2\pi x_2))/3. 
  \end{array}
  \label{eq:ex1}
\end{equation}
If $g(x) = \sum_{k \in \Omega} f(k) e^{2\pi \i x \cdot k/N}$ is the
(periodic) inverse Fourier transform of the input, this example models
the integration of $g$ over ellipses where $c_1(x)$ and $c_2(x)$ are
the axis lengths of the ellipse centered at the point $x\in X$. In
truth, the exact formula of this generalized Radon transform contains
an amplitude term $a(x,k)$ involving Bessel functions of the first and
second kinds. Nonetheless, we wish to focus on the main computational
difficulty, the highly oscillatory phase in this example, and simply
set the amplitude $a(x,k)$ to one. Table \ref{tbl:num1} summarizes the
results of this example for different combinations of the grid size
$N$ (the grid is $N \times N$) and of the degree of the polynomial
interpolation $q$.

\begin{table}[h]
  \begin{center}
    \begin{tabular}{|ccccc|}
      \hline
      $(N,q)$ & $T_a$(sec) & $T_d$(sec) & $T_d/T_a$ & $\eps_a$\\
      \hline
      %(64,5)  & 1.89e+0 & 8.00e-1 & 4.23e-1 & 1.12e-2\\
      (256,5) & 6.11e+1 & 3.20e+2 & 5.24e+0 & 1.26e-2\\
      (512,5) & 2.91e+2 & 5.59e+3 & 1.92e+1 & 1.56e-2\\
      (1024,5)& 1.48e+3 & 9.44e+4 & 6.37e+1 & 1.26e-2\\
      (2048,5)& 6.57e+3 & 1.53e+6 & 2.32e+2 & 1.75e-2\\
      (4096,5)& 3.13e+4 & 2.43e+7 & 7.74e+2 & 1.75e-2\\
      \hline
      %(64,7)  & 3.91e+0 & 8.00e-1 & 2.05e-1 & 6.39e-4\\
      (256,7) & 1.18e+2 & 3.25e+2 & 2.76e+0 & 7.57e-4\\
      (512,7) & 5.54e+2 & 5.47e+3 & 9.87e+0 & 6.68e-4\\
      (1024,7)& 2.76e+3 & 9.48e+4 & 3.44e+1 & 6.45e-4\\
      (2048,7)& 1.23e+4 & 1.46e+6 & 1.19e+2 & 8.39e-4\\
      (4096,7)& 5.80e+4 & 2.31e+7 & 3.99e+2 & 8.18e-4\\
      \hline
      %(64,9)  & 1.04e+1 & 8.00e-1 & 7.70e-2 & 2.50e-5\\
      (256,9) & 2.46e+2 & 3.10e+2 & 1.26e+0 & 3.15e-5\\
      (512,9) & 1.03e+3 & 5.19e+3 & 5.06e+0 & 3.14e-5\\
      (1024,9)& 4.95e+3 & 9.44e+4 & 1.91e+1 & 3.45e-5\\
      (2048,9)& 2.21e+4 & 1.48e+6 & 6.71e+1 & 4.01e-5\\
      (4096,9)& 1.02e+5 & 2.23e+7 & 2.18e+2 & 4.21e-5\\
      \hline
      %(64,11) & 2.12e+1 & 8.00e-1 & 3.78e-2 & 8.12e-7\\
      (256,11)& 4.66e+2 & 3.07e+2 & 6.59e-1 & 7.34e-7\\
      (512,11)& 1.69e+3 & 4.53e+3 & 2.68e+0 & 7.50e-7\\
      (1024,11)&8.33e+3 & 9.50e+4 & 1.14e+1 & 5.23e-7\\
      (2048,11)&3.48e+4 & 1.49e+6 & 4.27e+1 & 5.26e-7\\
      \hline
    \end{tabular}
  \end{center}
  \caption{Computational results with the phase function given by \eqref{eq:ex1}.
    $N\times N$ is the size of the domain;
    $q$ is the size of the Chebyshev interpolation grid in each dimension;
    $T_a$ is the running time of the algorithm in seconds; $T_d$ is
    the estimated running time of the direct evaluation method and 
    $T_d/T_a$ is the speedup factor; finally, $\eps_a$ is the accuracy estimated
    with \eqref{eq:err}.
  }
  \label{tbl:num1}
\end{table}

Next, we use the algorithm described at the end of Section
\ref{sec:algo} to study the performance in the more general setup of
variable amplitudes \eqref{eq:allfio}. The second example is the exact
formula for integrating over circles with radii $c(x)$ centered at
the points $x\in X$
\[
u(x) = 
\sum_{k\in \Omega} a_+(x,k) e^{2\pi\i \Phi_+(x,k)} f(k) + 
\sum_{k\in \Omega} a_-(x,k) e^{2\pi\i \Phi_-(x,k)} f(k)
\]
where the amplitudes and phases are given by
\begin{eqnarray}
  && a_\pm(x,k) = \left(J_0(2\pi c(x)|k|) \pm i Y_0(2\pi c(x) |k| \right) \, e^{\mp 2\pi i c(x)|k|} \nonumber\\
  && \Phi_\pm(x,k) = x\cdot k + c(x)|k|   \label{eq:ex2}\\
  && c(x) = (3 + \sin(2\pi x_1) \sin(2\pi x_2))/4. \nonumber
\end{eqnarray}
(Above the functions $J_0$ and $Y_0$ are special Bessel functions. The
Appendix in \cite{candes-2007-fcfio} details the derivation of these
formulas). We use the randomized procedure described in
\cite{candes-2007-fcfio} to construct the low-rank separated
approximation for $a_\pm(x,k)$. For an accuracy of 1e-7, the resulting
approximation contains only 3 terms. Table \ref{tbl:num2} summarizes
the results of this example for different combinations of $N$ and $q$.

\begin{table}[h]
  \begin{center}
    \begin{tabular}{|ccccc|}
      \hline
      $(N,q)$ & $T_a$(sec) & $T_d$(sec) & $T_d/T_a$ & $\eps_a$\\
      \hline
      %(64,5)  & 4.29e+0 & 1.20e+1 & 2.80e+0 & 1.34e-2\\
      (256,5) & 1.39e+2 & 3.20e+3 & 2.31e+1 & 1.48e-2\\
      (512,5) & 7.25e+2 & 5.20e+4 & 7.17e+1 & 1.62e-2\\
      (1024,5)& 3.45e+3 & 8.34e+5 & 2.42e+2 & 1.90e-2\\
      \hline
      %(64,7)  & 8.97e+0 & 1.18e+1 & 1.32e+0 & 4.95e-4\\
      (256,7) & 2.69e+2 & 3.21e+3 & 1.19e+1 & 4.71e-4\\
      (512,7) & 1.38e+3 & 5.20e+4 & 3.78e+1 & 7.30e-4\\
      (1024,7)& 6.43e+3 & 8.35e+5 & 1.30e+2 & 6.35e-4\\
      \hline
      %(64,9)  & 1.85e+1 & 1.20e+1 & 6.50e-1 & 2.33e-5\\
      (256,9) & 5.23e+2 & 3.20e+3 & 6.12e+0 & 1.59e-5\\
      (512,9) & 2.49e+3 & 5.17e+4 & 2.08e+1 & 2.97e-5\\
      (1024,9)& 1.15e+4 & 8.32e+5 & 7.25e+1 & 1.75e-5\\
      \hline
      %(64,11) & 4.68e+1 & 1.18e+1 & 2.53e-1 & 7.37e-7\\
      (256,11)& 1.04e+3 & 3.18e+3 & 3.06e+0 & 8.03e-7\\
      (512,11)& 4.10e+3 & 5.11e+4 & 1.24e+1 & 9.38e-7\\
      (1024,11)&1.84e+4 & 8.38e+5 & 4.57e+1 & 8.01e-7\\
      \hline
    \end{tabular}
  \end{center}
  \caption{Computational results with the amplitudes and phase functions 
    given by 
    \eqref{eq:ex2}.
  } 
  \label{tbl:num2}
\end{table}

From these tables, the first observation is that the accuracy is well
controlled by the size of the Chebyshev grid, and that the estimated
accuracy $\eps_a$ improves on average by a factor of 30 every time $q$
is increased by a factor of 2.  In practical applications, one often
specifies the accuracy $\eps_a$ instead of the grid size $q$. To adapt
to this situation, the quantity \eqref{eq:err} can be used to estimate
the error; whenever the error is too large, one can simply increase
the value of $q$ until the desired accuracy is reached. The second
observation is that the accuracy decreases only slightly when $N$
increases, indicating that the algorithm is numerically stable. This
is due to the fact that the Lebesgue constant of the Chebyshev
interpolation is almost optimal, i.~e.~the Chebyshev interpolation
operator has almost the minimum operator norm among all Lagrange
interpolants of the same order \cite{rivlin-1974-cp}.

These results show that the empirical running time of our algorithm
closely follows the $O(N^2 \log N)$ asymptotic complexity. Each time
we double $N$, the size of the grid quadruples. The
corresponding running time and speedup factor increase by a factor
roughly equal to 4 as well. We note that for large values of $N$ which
are of interest to us and to practitioners, the numerical results show
a very substantial speedup factor over direct evaluation. For
instance, for $4,096 \times 4,096$ grids, we gain three order of
magnitudes since one can get nearly two digits of accuracy with a
speedup factor exceeding 750.

The article \cite{candes-2007-fcfio} proposed an $O(N^{2.5}\log N)$
approach based on the partitioning of the frequency domain into
$\sqrt{N}$ conical region. Though the time complexity of this former
algorithm may not be optimal, we showed that it was efficient in parts
because its main computational component, the nonuniform fast Fourier
transform, is highly optimized.  Comparing Tables 4 and 5 in
\cite{candes-2007-fcfio}\footnote{The results in Tables 4 and 5 of
  \cite{candes-2007-fcfio} were obtained on a desktop with a 2.6GHz
  CPU, which is slightly slower yet comparable to the computer used
  for the tests in this section. The implementation of the nonuniform
  fast Fourier transform in \cite{candes-2007-fcfio} was written in
  C++ and complied as a MEX-function. Finally, the two examples are
  not exactly similar but this slight difference is unessential.}
with the numerical results presented here, we observe that both
approaches take roughly the same time for $N=256$ and $512$. For $N
\le 256$, the approach based on conical partitioning is faster as its
complexity has a smaller constant. For $N \ge 512$, however, the
current approach based on the butterfly algorithm clearly outperforms
our former approach.

It is straightforward to generalize our algorithm to higher
dimensions. In three dimensions for example, the main modification is
to use a three dimensional Chebyshev grid to interpolate $e^{2\pi\i
  R^{AB}(x,k)}$.  Consider again a simple 3D example modeling the
integration over spheres with varying radii in which the now
6-dimensional phase function $\Phi(x,k)$, $x, k \in \R^3$, is given by
\begin{equation}
  \Phi(x,k) = x\cdot k + c(x) |k|, \qquad 
  c(x)=(3 + \sin(2 \pi x_1)\sin(2 \pi x_2)\sin(2 \pi x_3))/4. \nonumber
  \label{eq:ex3}
\end{equation}
Our 3D numerical results are reported in Table \ref{tbl:num3}. In this
setup, we see that our approach offers a significant speedup even for
moderate values of $N$.
\begin{table}[h]
  \begin{center}
    \begin{tabular}{|ccccc|}
      \hline
      $(N,q)$ & $T_a$(sec) & $T_d$(sec) & $T_d/T_a$ & $\eps_a$\\
      \hline
      (64,7)  & 1.79e+3 & 7.33e+3 & 4.10e+0 & 3.32e-3\\
      (128,7) & 1.58e+4 & 4.77e+5 & 3.02e+1 & 4.06e-3\\
      (256,7) & 1.44e+5 & 2.97e+7 & 2.06e+2 & 3.96e-3\\
      \hline
    \end{tabular}
  \end{center}
  \caption{Computational results in 3 dimensions with the phase function given by \eqref{eq:ex3}.}
  \label{tbl:num3}
\end{table}

%----------------------------------------------------------
\section{Conclusions and Discussions}
\label{sec:concl}

This paper introduced a novel and accurate algorithm for evaluating
discrete FIOs.  Underlying this approach is a key mathematical
property, which says that the kernel, restricted to special subdomains
in time and frequency, is approximately of very low-rank. Our strategy
operationalizes this fact by using a multiscale partitioning of the
time and frequency domain together with the butterfly structure to
achieve an $O(N^2 \log N)$ asymptotic complexity.

A different way to achieve a near-optimal $O(N^2\log N)$ complexity
might be to use the curvelet transform
\cite{candes-2004-nfc,candes-2006-fdct} of Cand\`es and Donoho, or the
wave atoms \cite{demanet-2007-wasop} of Demanet and Ying.  In 
\cite{candes-2003-cfio,candes-2005-crwpos}, Cand\`{e}s and Demanet
proved that the curvelet representation of FIOs is optimally sparse
(the wave atom representation also offers the same optimality), a
property which relies on the role played by the second dyadic
decomposition of Stein and his collaborators
\cite{stein-1993-ha}. Whether one can operationalize this mathematical
insight into an efficient algorithm seems an interesting direction for
future research.

The geometric low-rank property together with the butterfly algorithm
appear to be a very powerful combination to obtain fast algorithms for
computing certain types of highly oscillatory integrals. We already
discussed the work of O'Neil and Rokhlin \cite{oneil-2007-ncabft} who
have used the butterfly algorithm to design fast special transforms,
and of Ying who has extended this approach to develop fast algorithms
for Fourier transforms with sparse data \cite{ying-2008-sftba} and
Fourier transforms with summation constraints
\cite{ying-2008-fcpft}. Clearly, it would be of interest to identify
wide classes of problems for which this general approach may prove
powerful. 

{\small

\subsection*{Acknowledgments}
E.~C.~is partially supported by the Waterman Award from the National
Science Foundation and by an ONR grant N00014-08-1-0749. L.~D.~is
partially supported by a National Science Foundation grant DMS-0707921. L.~Y.~is partially supported by an Alfred
P.~Sloan Fellowship and a National Science Foundation grant
DMS-0708014.

\bibliographystyle{abbrv}
\bibliography{ref}
}
\end{document}